\documentclass[10pt]{amsart}
\usepackage{amsfonts}
\usepackage{amsmath}
\usepackage{cases}
\usepackage{amssymb}
\usepackage{graphicx}
\usepackage{color}
\usepackage{bigstrut}

\newtheorem{maintheo}{Theorem}
\newtheorem{theorem}{Theorem}[section]
\newtheorem{proposition}[theorem]{Proposition}
\newtheorem{lemma}[theorem]{Lemma}
\newtheorem{corollary}[theorem]{Corollary}

\newtheorem{remark}[theorem]{Remark}
\newtheorem{claim}[theorem]{Claim}

\theoremstyle{definition}%
\def\pasdegrille{\let\grille = \pasgrille}

\def\aat#1#2#3{
\divide \dimen1 by 48 \dimen3=\dimen1 \multiply \dimen1 by #1
\advance \dimen1 by -\dimen3 \divide \dimen1 by 101 \multiply
\dimen1 by 100 \divide \dimen2 by \count11 \multiply \dimen2 by #2
\setbox0=\hbox{#3}\ht0=0pt\dp0=0pt
  \rlap{\kern\dimen1 \vbox to0pt{\kern-\dimen2\box0\vss}}\dimen1= \wd1
\dimen2=\ht1}
\def\pasgrille{
\count12= \dimen1 \divide \count12 by 50 \divide \dimen2 by \count12
\count11 =\dimen2 \ \divide \dimen1 by 48
\setlength{\unitlength}{\dimen1} \smash{\rlap{\ }} \dimen1= \wd1
\dimen2=\ht1 }
\def\grille{
\count12= \dimen1 \divide \count12 by 50 \divide \dimen2 by \count12
\count11 =\dimen2 \ \divide \dimen1 by 48
\setlength{\unitlength}{\dimen1}
\smash{\rlap{\graphpaper[1](0,0)(50, \count11)}} \dimen1= \wd1
\dimen2=\ht1 }

\pasdegrille

\numberwithin{equation}{section}

\newcommand{\R}{\mathbb R}

\newcommand{\beq}{\begin{equation}}
\newcommand{\eeq}{\end{equation}}
\newcommand{\ben}{\begin{eqnarray}}
\newcommand{\een}{\end{eqnarray}}
\newcommand{\beno}{\begin{eqnarray*}}
\newcommand{\eeno}{\end{eqnarray*}}

\def\eps{\varepsilon}

\title[Nonlinear wave equation]{Blow-up of a critical Sobolev norm for energy-subcritical and energy-supercritical wave equations}
\date{\today}

\author{Thomas Duyckaerts$^1$}
\address{
Universit\'e Paris 13,
Sorbonne Paris Cit\'e, LAGA (UMR CNRS 7539),
99, Avenue J.-B. Cl\'ement,
F-93430 Villetaneuse}
\thanks{$^1$Institut Universitaire de France and LAGA, Universit\'e Paris 13, Sorbonne Paris Cit\'e. Partially supported by ERC advanced grant no 291214 BLOWDISOL}
\email{duyckaer@math.univ-paris13.fr}
 \author{Jianwei Yang$^2$}
\thanks{$^2$LAGA, Universit\'e Paris 13, Sorbonne Paris Cit\'e and Beijing International Center for Mathematical Research, Peking
University. Partially supported by ERC advanced grant no 291214 BLOWDISOL}
\address{Beijing International Center for Mathematical Research, Peking University, Beijing 100871, China}
\email{geewey\_young@pku.edu.cn}
\keywords{Supercritical wave equation, Strichartz estimates, scattering, blow-up, profile decomposition}
\subjclass[2010]{Primary 35L71; Secondary 35B40, 35B44}

\newcommand{\dotW}{\dot{\mathcal{W}}^{1,m}}

\newcommand{\AAA}{\mathcal{A}}
\newcommand{\HHH}{\mathcal{H}}
\newcommand{\JJJ}{\mathcal{J}}
\newcommand{\LLL}{\mathcal{L}}
\newcommand{\NNN}{\mathcal{N}}
\newcommand{\TTT}{\mathcal{T}}

\newcommand{\mfkE}{\mathfrak{E}}
\newcommand{\mfktU}{\widetilde{\mathfrak{U}}}
\newcommand{\mfkR}{\mathfrak{R}}

\newcommand{\indic}{1\!\!1}

\newcommand{\tg}{\tilde{g}}
\newcommand{\tu}{\tilde{u}}

\newcommand{\EMPH}[1]{\medskip\noindent\textit{#1}.}

\newcommand{\ds}{\displaystyle}

\newcommand{\RR}{\mathbb{R}}
\newcommand{\NN}{\mathbb{N}}

\newcommand{\tU}{\widetilde{U}}

\DeclareMathOperator{\rad}{rad}
\DeclareMathOperator{\esssupp}{ess\,supp}

\usepackage{mathtools}

\begin{document}
\begin{abstract}
This work concerns linear and nonlinear wave equations in three space dimensions in a radial setting.
We first prove new Strichartz estimates for the linear equation, in weighted Sobolev spaces which were introduced in a preceding article of Tristan Roy and the first author. These spaces are not Hilbert space. However they are local, thus adapted to finite speed of propagation, and related to a conservation law of the linear wave equation. We also construct the adapted profile decomposition. 
Our main motivation is to study the radial semilinear wave equation in three space dimensions with a power-like nonlinearity which is greater than cubic, and not quintic. We prove that the equation is locally well-posed in the scale-invariant weighted Sobolev space mentioned above, and that the norm of any non-scattering solution goes to infinity at the maximal time of existence. This gives a refinement on known results on energy-subcritical and energy-supercritical wave equation, with a unified proof. 
\end{abstract}

\maketitle

\section{Introduction}
\subsection{Motivation and background}
Consider the semilinear wave equation in $1+3$ dimensions
\begin{equation}
\label{NLW}
(\partial_{t}^{2}-\Delta)u=\iota |u|^{2m}u,
\end{equation}
with initial data
\begin{equation}
 \label{ID}
 u(0,x)= u_{0}(x),\;\partial_{t}u(0,x)=u_{1}(x),
\end{equation} 
where $x\in \R^3$ and $t\in \RR$. The parameters $m>1$ and $\iota\in \{\pm 1\}$ are fixed. The equation is \emph{focusing} when $\iota=1$ and \emph{defocusing} when $\iota=-1$. It has the following scaling invariance: if $u(t,x)$ is a solution of \eqref{NLW} and $\lambda>0$, then $\lambda^{\frac{1}{m}}u(\lambda t,\lambda x)$
is also a solution. It is well-posed in the scale invariant Sobolev space $\dot{\HHH}^{s_c}:=\dot{H}^{s_c}(\RR^3)\times \dot{H}^{s_c-1}(\RR^3)$, where $s_c=\frac 32-\frac 1m$ is the critical Sobolev exponent. Equation \eqref{NLW} is \emph{energy-subcritical} if $s_c<1$ (equivalently $m<2$), \emph{energy-critical} if $s_c=1$ ($m=2$) and \emph{energy-supercritical} if $s_c>1$ ($m>2$).\medskip

The dynamics of \eqref{NLW} depend in a crucial way on the value of $m$ and the sign of $\iota$.

The energy-critical case $m=2$ is particular. The conserved energy
\begin{equation*}
 E(\vec{u}(t))=\frac{1}{2}\int |\nabla u(t,x)|^2\,dx+\frac{1}{2}\int (\partial_tu(t,x))^2\,dx-\frac{\iota}{2m+2}\int |u(t,x)|^{2m+2}\,dx
\end{equation*} 
is well-defined in $\dot{\HHH}^{s_c}=\dot{\HHH}^1=\dot{H}^1\times L^2$. When the nonlinearity is defocusing, the conservation of the energy implies that all solutions are bounded in $\dot{\HHH}^1$. It was proved in the $90$s that all solutions are global and scatter to a linear solution in the energy space, i.e. that there exists a solution $u_{\rm L}$ of the linear wave equation:
\begin{equation}
\label{LW}
(\partial_{t}^{2}-\Delta)u_{\rm L}=0,\;(t,x)\in \mathbb R\times \mathbb R^{3},
\end{equation}
with initial data in $\dot{\HHH}^1$, such that 
\begin{equation}
 \label{scatteringH1}
\lim_{t\to+\infty} \left\|\vec{u}(t)-\vec{u}_{\rm L}(t)\right\|_{\dot{\HHH}^1}=0.
\end{equation} 
(see \cite{Grillakis90b, Grillakis92, GiSoVe92, ShSt93,ShSt94, Kapitanski94, GiVe95, Nakanishi99b, BaSh98}). In the focusing case, there exist solutions that do not scatter. Indeed, there exist solutions of \eqref{NLW} that blow up in finite time with a \emph{type I} behavior, \emph{i.e.} such that:
\begin{equation*}
 \lim_{t\to T_+(u)}\|\vec{u}(t)\|_{\dot{\HHH}^1}=+\infty,
\end{equation*} 
where $T_+(u)$ is the maximal time of existence of $u$. Furthermore, the equation also admits stationary solutions and more generally traveling waves. It was proved in \cite{DuKeMe13} that any radial solution that does not scatter and is not a type I blow-up solution decouples asymptotically as a sum of rescaled stationary solutions and a dispersive term. This includes global non-scattering solutions (see \cite{KrSc07,DoKr13}, and also \cite{MaMe16,Jendrej16P} in higher space dimensions for examples of such solutions) and solutions that blow up in finite time but remain bounded in the energy space, called \emph{type II blow-up} solutions  (see e.g. \cite{KrScTa09,KrSc14} and, in higher dimensions \cite{HiRa12,Jendrej17}).
\medskip

The case $m\neq 2$ is quite different. It is known that stationary solutions do not exist in the critical Sobolev space, even for focusing nonlinearity (see e.g. \cite{JosephLundgren72}, \cite[Theorem 2]{Farina07}), and it is conjectured that any solution that does not satisfy:
\begin{equation}
\label{Norm_BlowUp}
\lim_{t\to T_+(u)} \|\vec{u}(t)\|_{\dot{\HHH}^{s_c}}=+\infty
\end{equation} 
is global and scatters to a linear solution for positive times. A slightly weaker version of this result was proved in many works, namely that 
if the solution does not scatter, then 
\begin{equation}
\label{Norm_BlowUpBis}
\limsup_{t\to T_+(u)} \|\vec{u}(t)\|_{\dot{\HHH}^{s_c}}=+\infty.
\end{equation} 
See \cite{KeMe11,DuKeMe12c} (radial case, $m>2$), \cite{Shen13} (radial case, $1<m<2$, see also \cite{Rodriguez16P}), \cite{KiVi11} (defocusing nonradial case, $m>2$), \cite{DodsonLawrie15} (radial case, $m=1$), and also
\cite{KiStVi14} for the nonradial defocusing case, $1\leq m<2$, where \eqref{Norm_BlowUpBis} is proved for finite time blow-up solutions with initial data in the energy space. 

Note that none of the preceding works excludes the existence of a nonscattering solution of \eqref{NLW} such that 
$$\limsup_{t\to T_+(u)}\|\vec{u}(t)\|_{\dot{\HHH}^{s_c}}=+\infty\text{ and }\liminf_{t\to T_+(u)}\|\vec{u}(t)\|_{\dot{\HHH}^{s_c}}<\infty.$$
In \cite{DuRoy15P}, this type of solution was ruled out in the case $m>2$: for any radial nonscattering solution of the equation, the critical Sobolev norm goes to infinity as $t\to T_+(u)$. 
\medskip

It is interesting to compare the theorems cited above with analogous ones for other equations, and in particular for the nonlinear Schr\"odinger equation:
\begin{equation}
\label{NLS}
 i\partial_t v-\Delta v=\iota |v|^{2m}v,
\end{equation} 
 For the defocusing equation ($\iota=-1$), the fact that the bound of a critical norm implies scattering is known in the cubic case in three space dimensions (see \cite{KeMe10}) and in energy-supercritical cases in large space dimensions (see \cite{KiVi10b}). In \cite{MeRa08} Merle and Rapha\"el considered the focusing equation \eqref{NLS} with $\iota=1$ and an $L^2$ supercritical (i.e. pseudo-conformally supercritical), energy subcritical nonlinearity, that is $\frac{2}{3}<m<2$ when the space dimensions is three. This condition is the analogue of the condition $1<m<2$ (conformally supercritical and energy subcritical power) for the wave equation.
They proved that if $u$ is radial with initial data in the intersection of $\dot{H}^1$ and the critical Sobolev space, and if $T_+(v)$ is finite, then 
$$\|v(t)\|_{L^{3m}}\geq \frac{1}{C}|\log(T_+(v)-t)|^{\alpha},$$
for some constant $\alpha>0$. Note that in this case there exists global, bounded, nonscattering solution. The space $L^{3m}$ is scale-invariant and strictly larger than the critical Sobolev space.
Analogous results are known for Navier-Stokes equations (see \cite{EsSeSv03}, \cite{KenigKoch11}, \cite{Seregin12}, \cite{GaKoPl13} and \cite{GaKoPl16}). For example it is proved in \cite{Seregin12} that the scale invariant $L^3$ norm of a solution blowing-up in finite time goes to infinity at the blow-up time.
\medskip

Going back to equation \eqref{NLW} with $m\neq 2$, many questions remain open:
\begin{itemize}
 \item Is it true that all non-scattering solution of \eqref{NLW} satisfy \eqref{Norm_BlowUp} in the nonradial case, or if $1<m<2$?
 \item Can one lower the regularity of the scale-invariant norm used in \eqref{Norm_BlowUp}, as in the case of nonlinear Schr\"odinger and Navier-Stokes equations?
 \item Is it possible to give an explicit lower-bound of the critical norm, in the spirit of the article of Merle and Rapha\"el \cite{MeRa08}?
\end{itemize}
In this article, we give a partial answer to the first two questions in the radial case.
This is based on a new well-posedness theory for equation \eqref{NLW}, in a  scale invariant weighted Sobolev space $\LLL^m$ which is not Hilbertian, but is related to a conserved quantity of the linear wave equation and is compatible with the finite speed of propagation.

\subsection{Strichartz estimates and local well-posedness}
Consider the following norm for radial functions $(u_0,u_1)$ on $\RR^3$:
$$\|(u_0,u_1)\|_{\LLL^m}=\left(\int_{0}^{+\infty}\left( |r\partial_r u_0|^m+|r u_1|^m\right)\,dr\right)^{\frac{1}{m}},$$
and define the space $\LLL^m$ as the closure of radial, smooth, compactly supported functions for this norm. Note that $\LLL^2$ is exactly\footnote{In all the article, the index $\rad$ denotes the subspace of radial elements of a given space of distributions on $\RR^3$} $\dot{\HHH}^1_{\rad}$. The $\LLL^m$ norm was introduced in \cite{DuRoy15P}, in the case $m>2$, as a scale-invariant substitute to the energy norm $\dot{H}^1\times L^2$ norm. Let us mention that $\dot{\HHH}^{s_c}_{\rad}\subset \LLL^m$ if $m>2$, and $\LLL^m\subset \dot{\HHH}^{s_c}_{\rad}$ if $1<m<2$ (see Proposition \ref{P:prelim} below). It was observed in \cite{DuRoy15P} that the $\LLL^m$ norm is almost conserved for solutions of the linear wave equation: we will indeed introduce in Section \ref{S:LWP} a \emph{conserved} quantity (the generalized energy) which is equivalent to this norm. We first prove Strichartz estimates for the linear wave equation. If $I$ is a real interval, we denote by $S(I)$ the space defined by the norm:
$$\|f\|_{S(I)}=\left(\int_{I} \left(  \int_{0}^{+\infty} |f(t,r)|^{(2m+1)m}r^mdr\right)^{\frac{1}{m}}dt\right)^{\frac{1}{2m+1}}.$$
\begin{maintheo}
\label{T:Strichartz}
 Let $v$ be a solution of the linear wave equation
 $$ \partial_t^2v-\Delta v=0,\quad (v,\partial_tv)_{\restriction t=0}=(v_0,v_1)\in \LLL^m.$$
 Then $v\in S(\R)$ and
$$\|v\|_{S(\R)}\leq C\|(v_0,v_1)\|_{\LLL^m}.$$
\end{maintheo}
Note that Theorem \ref{T:Strichartz} generalizes, in the radial case, the $L^5L^{10}$ Strichartz/Sobolev estimate for finite-energy solutions of the linear wave equation to the case $m\neq 2$. Let us mention that we prove more general Strichartz estimates, including estimates for the nonhomogeneous wave equation  (see Subsection \ref{SS:Strichartz} for the details). As a consequence, we obtain local well-posedness in $\LLL^m$ for equation \eqref{NLW}:
\begin{maintheo}
 \label{T:LWP}
 For $m>1$,
 equation \eqref{NLW} is locally well-posed in $\LLL^m$. For any initial data $(u_0,u_1)$ in $\LLL^m$, there exists a unique solution $u$ of \eqref{NLW}, \eqref{ID} defined on a ma\-xi\-mal interval of existence $I_{\max}(u)=(T_-(u),T_+(u))$ such that $\vec{u}\in C^0(I_{\max}(u),\LLL^m)$ and for all compact interval $J\Subset I_{\max}(u)$,  $u\in S(J)$. Furthermore,
 $$ T_+(u)<\infty\Longrightarrow \|u\|_{S([0,T_+(u)))}=+\infty.$$
\end{maintheo}

We obtain Theorem \ref{T:Strichartz} and the other generalized Strichartz estimates of Subsection \ref{SS:Strichartz} by interpolating between the known generalized Strichartz estimates of Ginibre and Velo \cite{GiVe95} (see also \cite{LiSo95})  in correspondence to the case $m=2$, and Strichartz-type estimates obtained by a new method, based on the continuity of the Hardy-Littlewood maximal function from $L^1$ to $L^1_w$ (see Subsection \ref{SS:Strichartz}). 

We also construct a profile decomposition for sequences of functions that are bounded in $\LLL^m$, which is adapted to equation \eqref{NLW}, in the spirit of the one of Bahouri-G\'erard \cite{BaGe99} which corresponds to the case $m=2$. This construction is based on a refined Sobolev embedding due to Chamorro \cite{Chamorro11}. The fact that $\LLL^m$ is not a Hilbert space yields a new technical difficulty, namely that the usual Pythagorean expansion of the norm does not seem to be valid and must be replaced by a weaker statement, closer to Bessel's inequality than to Pythagorean Theorem.
We refer to Solimini \cite{Solimini95} and Jaffard \cite{Jaffard99} for other non-Hilbertian profile decompositions where this type of inequalities also appears.

The definition of the space $\LLL^m$ does not involve any fractional derivative and is technically easier to handle than the space $\dot{\HHH}^{s_c}$ with $m\neq 2$, where the latter are all defined by norms that are not compatible with finite speed of propagation. We hope that the Strichartz estimates and profile decomposition proved in this article will find applications for nonlinear wave equations apart from \eqref{NLW}.

\subsection{Blow-up of the critical Sobolev norm for the nonlinear equation}

Our second result is that the dichotomy proved in \cite{DuRoy15P} remains valid in $\LLL^m$, as long as $m\neq 2$:
\begin{maintheo}
\label{T:scattering}
Assume $m>1$ and $m\neq 2$.
Let $u$ be a radial solution of \eqref{NLW}, \eqref{ID}, with $(u_0,u_1)\in \LLL^m$ and maximal positive time of existence $T_+$. Then one of the following holds:
\begin{enumerate}
 \item $\ds \lim_{t\to T_+(u)} \|\vec{u}(t)\|_{\LLL^m}=+\infty$ or
\item $T_+(u)=+\infty$ and $u$ scatters forward in time to a linear solution, i.e. there exists a solution $u_{\rm L}$ of \eqref{LW}, with initial data $\LLL^m$, such that
$$\lim_{t\to+\infty}\left\|\vec{u}(t)-\vec{u}_{\rm L}(t)\right\|_{\LLL^m}=0.$$
\end{enumerate}
\end{maintheo}
In the energy-supercritical case $m>2$, Theorem \ref{T:scattering} improves the result of \cite{DuRoy15P} since $\dot{\HHH}^{s_c}$ is continuously embedded into $\LLL^m$. In the case $1<m<2$, $\LLL^m$ is continuously embedded into $\dot{\HHH}^{s_c}$ and Theorem \ref{T:scattering} is not strictly stronger that the result of \cite{Shen13}. However, Theorem \ref{T:scattering} is also new, since it says that as least some scale invariant norm of $u$ must go to infinity as $t$ goes to $T_+(u)$. It is very natural to conjecture that the $\dot{\HHH}^{s_c}$ norm of the solution also goes to infinity but this is still an open question.

Once the Strichartz estimates, well-posed theory and profile decomposition in $\LLL^m$ are known, the proof of Theorem \ref{T:scattering} (sketched in Sections \ref{S:exterior}, \ref{S:dispersive} and \ref{S:endofproof})  is very close to the proof of the corresponding result in \cite{DuRoy15P}, with some simplifications due to the use of the space $\LLL^m$ instead of $\dot{\HHH}^{s_c}$ in all the proof. As in \cite{DuRoy15P}, we use the \emph{channels of energy method} initiated in \cite{DuKeMe11a}, and the main ingredient of the proof is an exterior energy estimate for radial solutions of the linear wave equation for the $\LLL^m$-energy, which generalizes the exterior energy estimate used in \cite{DuKeMe11a,DuKeMe12c,DuKeMe13}.

According to Theorem \ref{T:scattering}, there are three potential types of dynamics for equation \eqref{NLW}: scattering, finite time blow-up solutions such that the critical norm goes to infinity at the blow-up time,  and global solutions such that the critical norm goes to infinity as $t$ goes to infinity. Only two of these dynamics are known to exist: scattering (for both focusing and defocusing nonlinearities) and finite time blow-up (for focusing nonlinearity only). Indeed, in the focusing case, it is possible to construct blow-up solutions with smooth, compactly supported, initial data using finite speed of propagation and the ordinary differential equation  $y''=|y|^{2m}y$. Another type of blow-up solution was constructed by C. Collot in \cite{Collot14P}, for some energy-supercritical nonlinearity in large space dimension: in this case the scale-invariant Sobolev norms blow up logarithmically.
\medskip

It is natural to conjecture that all solutions in $\LLL^m$ are global in the defocusing case. For $m<2$, this follows from conservation of the energy if the data is assumed to be in $\dot{\HHH}^1$, and only the case of low-regularity solution is open. For supercritical nonlinearity $m>2$, it is a very delicate issue even for smooth initial data, as the recent construction by T.~Tao of a a finite time blow-up solution for a defocusing \emph{system}\footnote{The unknown $u$ is $\RR^{40}$ valued} of energy supercritical wave equation suggests \cite{TaoBUp2016}.
\medskip

The existence of global solutions blowing-up at infinity with initial data in $\LLL^m$ (or $\dot{\HHH}^{s_c}$) is also completely open. We refer to \cite{KrSc14P} for the construction of global, smooth, non-scattering solutions in the case $m=3$. The initial data of these solutions do not belong
either to the critical Sobolev spaces $\dot{\HHH}^{7/6}$ or to the $\LLL^3$ space (but are, however, in all spaces $\dot{\HHH}^{s}$, $s>7/6$). This construction and Theorem \ref{T:scattering} seems to suggest that any global solution with initial data decaying sufficiently at infinity actually scatters, but we do not know of any rigorous result in this direction.

Let us finally mention the recent preprint of Beceanu and Soffer \cite{BeceanuSoffer16P} on equation \eqref{NLW} with supercritical nonlinearity $m>2$ where global existence is proved for a class of outgoing initial data.
\medskip

The outline of the paper is as follows: in Section \ref{S:LWP}, we prove the Strichartz estimate for the linear wave equation and deduce the Cauchy theory for equation \eqref{NLW}. In Section \ref{S:profiles}, we construct the profile decomposition. In Section \ref{S:exterior}, we prove the exterior energy property for nonzero solutions of equation \eqref{NLW} which is the core of the proof of Theorem \ref{T:scattering}. In Section \ref{S:dispersive}, we introduce the radiation term (i.e. the disperive part) of a solution which is bounded in the critical space for a sequence of times. In Section \ref{S:endofproof}, we conclude the proof.

\subsection*{Notations}
If $a$ and $b$ are two positive quatities, we write $a\lesssim b$ when there exists a constant $C>0$ such that $a\leq C b$ where the constant will be clear from the context. When the constant depends on some other quantity $M$, we emphasize the dependence by writing $a\lesssim_M b$.
We will write $a\approx b$ when we have both $a\lesssim b$ and $b\lesssim a$. We will write $a\ll b$ (resp. $a\gg b$) if there exists a sufficiently  large constant
$C>0$ such that $a\leq C b$ (resp. $a\geq Cb$).
We use $\mathcal{S}(\mathbb{R}^d)$
to denote the Schwartz class of functions on the Euclidean space $\mathbb{R}^d$.
\\
\\
If $f$ is a radial function depending on $t$ and $r:=|x|$, let
\begin{equation*}
\vec{f}:= (f,\partial_{t} f)\quad
\text{and}\quad
[f]_\pm(t,r)=(\partial_r\pm\partial_t)(rf).
\end{equation*}
Given $s \geq 0$ and $n$ a positive integer, we define
\begin{equation*}
\dot{\mathcal{H}}^{s}(\mathbb{R}^{n}) := \dot{H}^{s} (\mathbb{R}^{n}) \times \dot{H}^{s-1}(\mathbb{R}^{n}),
\end{equation*}
where $\dot{H}^{s}$ denotes the standard homogeneous Sobolev space. 
We let $L^p_t(I,L^q_x)$ be the space of measurable functions $f$ on $I\times \RR^3$ such that
$$ \|f\|_{L^p_t(I,L^q_x)}=\left(\int_I \left( \int_{\mathbb{R}^3}|f(t,x)|^qdx\right)^{\frac{p}{q}}\,dt  \right)^{1/p}<\infty.$$ 
Unless specified, the functional spaces ($L^p$, $\dot{H}^s$, etc...) are spaces of functions or distributions on $\R^3$ with the Lebesgue measure. On a measurable space $(\Omega,d\mu)$ where $\mu$ is positive,
the weak $L^q$ quasi-norm of a function $f$ is defined as
\[\|f\|_{L^q_w}:=\sup_{\lambda>0}\lambda\, \left(\mu\{x\in\Omega:|f(x)|>\lambda\}\right)^\frac{1}{q}\]
We shall also use the  weighted Lebesgue norm $L^q(\mathbb{R}^n,\omega)$ defined as
\[\|f\|_{L^q(\mathbb{R}^n,\omega)}
:=\left(\int_{\mathbb{R}^n}
|f(x)|^q\omega(x)dx
\right)^{\frac{1}{q}}\]
for some measurable function $\omega(x)$ as a weight.
For $q>1$, we use $q'=\frac{q}{q-1}$
to mean its Lebesgue conjugate. 
\\
\\
We denote by $\TTT_{R}$ the operator
\begin{equation}
\begin{array}{ll}
f & \mapsto \TTT_{R}(f) :=
\left\{
\begin{array}{ll}
f(R), \, |x| \leq R \\
f(|x|), \,  |x| \geq R.
\end{array}
\right.
\end{array}
\nonumber
\end{equation}
\\
Let $S_{\rm L}(t)$ denote the linear propagator, \emph{i.e.}
\begin{equation*}
S_{\rm L}(t)(w_0,w_1) :=
\cos{(tD)}w_0+\frac{\sin{(tD)}}{D}w_1,\quad D=\sqrt{-\triangle}.
\end{equation*}
If $u$ is a function of $t$ and $r$, we will denote by $F(\partial_{r,t}u)$ the sum $F(\partial_ru)+F(\partial_tu)$ e.g. $|\partial_{t,r}u|^m:=|\partial_tu|^m+|\partial_ru|^m$.

\subsection*{Acknowledgment}  The first author would like to thank Patrick G\'erard for poin\-ting out references \cite{Solimini95} and  \cite{Jaffard99}.

\section{Strichartz estimates and local well-posedness}
\label{S:LWP}
\subsection{Preliminaries}
For $m>1$, we denote  $\dotW$ as the closure of $C_{0,\rad}^{\infty}$ for the norm $\|\cdot\|_{\dotW}$ defined by:
$$\|\varphi\|_{\dotW}:=\left(\int_0^{+\infty} |\partial_r \varphi(r)|^mr^m\,dr  \right)^{1/m}.$$
\begin{proposition}\footnote{The proof is given in the appendix}
	\label{prop:fs}
	We have $f\in {\dotW}$ if and only if 
	$f(r)\in C^{0}_{\rm rad}\bigl((0,+\infty)\bigr)$ satisfies the conditions:
	\begin{equation}
	\label{eq:fs-1}
	\int_{0}^{+\infty}|r\partial_{r} f(r)|^{m}dr<+\infty\,,
	\end{equation}
	and 
	\begin{equation}
	\label{eq:fs-2}
	\lim_{r\rightarrow 0}r^{\frac1m}f(r)=
	\lim_{r\rightarrow \infty}r^{\frac1m}f(r)=0\,.
	\end{equation}
\end{proposition}

We denote $\LLL^m$ as the closure of $\left(C_{0,\rad}^{\infty}\right)^2$ for the norm $\|\cdot\|_{\LLL^m}$ below
$$\left\|(u_0,u_1)\right\|_{\LLL^m}:=\|u_0\|_{\dotW}+\left(\int_0^{+\infty} |u_1(r)|^mr^m\,dr  \right)^{1/m}.$$
Then:
\begin{proposition}
 \label{P:prelim}
 \begin{enumerate}
  \item \label{P1} If $m>2$ and $(u_0,u_1)\in \dot{\HHH}^{s_c}$, then $(u_0,u_1)\in \LLL^m$ and 
  $$\|(u_0,u_1)\|_{\LLL^m}\lesssim \|(u_0,u_1)\|_{\dot{\HHH}^{s_c}}.$$
 \item \label{P1'} If $1<m<2$ and $(u_0,u_1)\in \LLL^m$ then $(u_0,u_1)\in \dot{\HHH}^{s_c}$ and 
  $$\|(u_0,u_1)\|_{\dot{\HHH}^{s_c}}\lesssim \|(u_0,u_1)\|_{\LLL^m}.$$
 \item \label{P2} If $u_0\in \dotW$, then $u_0\in L^{3m}(\R^3)$ and 
  $$\|u_0\|_{L^{3m}}\lesssim \|u_0\|_{\dotW}.$$
 \item \label{P4} If $u_0\in \dotW$, and $R>0$, then
 $$ R|u_0(R)|^m +\int_R^{+\infty} |\partial_r(ru_0)|^m\,dr\approx \int_R^{+\infty} |\partial_r u_0|^mr^m\,dr,$$
 where the implicit constant does not depend on $R$. 
\end{enumerate}
\end{proposition}
\begin{proof}
For the proofs of properties \eqref{P1}, \eqref{P2}, \eqref{P4} see \cite[Lemma 3.2]{KeMe11} and \cite[Lemma 3.2 and 3.3]{DuRoy15P}. We prove \eqref{P1'} by duality from \eqref{P1}. Assume $m\in (1,2)$ and let $m'$ be the Lebesgue dual exponent of $m$. Let $(u_0,u_1)\in \LLL^m$ and $\varphi,\psi\in C^{\infty}_{0,\rad}(\RR^3)$. Note that 
$$\int_0^{\infty} r^2\partial_r u_0\partial_r \varphi\,dr=\int_0^{\infty} \partial_r(r u_0)\partial_r(r \varphi)\,dr.$$
By H\"older's inequality and \eqref{P1},
\begin{multline*}
 \left|\int_0^{\infty} r^2\partial_r u_0\partial_r \varphi\,dr\right|+\left|\int_0^{\infty} r^2u_1\psi\,dr\right|\leq \|(u_0,u_1)\|_{\LLL^m}\|(\varphi,\psi)\|_{\LLL^{m'}}\\
\leq \|(u_0,u_1)\|_{\LLL^m}\|(\varphi,\psi)\|_{\dot{\HHH}^{\frac{1}{2}+\frac{1}{m}}}.
\end{multline*}
This yields the announced result.
\end{proof}

Let $v(t,x)$ be a solution to the Cauchy problem
\begin{equation}
\label{1}
(\partial_{t}^{2}-\Delta) v(t,x)=0\,,\,
(v,\,\partial_{t}v)|_{t=0}=(v_{0},\,v_{1})\,,\,
t\in\mathbb{R},\,x\in \mathbb{R}^{3}\,,
\end{equation}
where the initial data is in $\LLL^m$. 
Denote by $r=|x|$ and set
\begin{equation}
\label{3}
F(\sigma)=\frac12\sigma\, v_{0}(|\sigma|)+\frac12\int_{0}^{|\sigma|}r\,v_{1}(r)\,dr.
\end{equation}
An explicit computation, using
\begin{equation}
 \label{eq_v}
(\partial_t^2-\partial_r^2)(rv)=0
\end{equation} 
 yields 
\begin{equation}
\label{2} v(t,r)=\frac1r\bigl(F(t+r)-F(t-r)\bigr).
\end{equation}
We have
 \begin{equation}
 \label{def_pm}
 [v]_+(t,r)=(\partial_r+\partial_t)(rv)=2\dot{F}(t+r),\quad [v]_-(t,r)=(\partial_r-\partial_t)(rv)=2\dot{F}(t-r).
 \end{equation}
 If $(v_0,v_1)\in \LLL^m$ we denote by
 $$E_m(v_0,v_1)=\int_0^{+\infty} \left( |\partial_r(rv_0)+r v_1|^m+ |\partial_r(rv_0)-r v_1|^m\right)\,dr,$$
 so that 
 $$E_m(\vec{v}(t))=\int_0^{+\infty} |[v]_+(t,r)|^m\,dr+\int_0^{+\infty} |[v]_-(t,r)|^m\,dr.$$
\begin{proposition}
\label{P:prelimLW}
Assume $1<m<+\infty$. 
Let $(v_0,v_1)\in \LLL^m$ and $v(t,r)$ be given by \eqref{1}.
 \begin{enumerate}
\item \label{P3} \emph{Equivalence of energy and $\LLL^m$ norm}.
 $$\left\|(v_0,v_1)\right\|_{\LLL^m}\approx \int_{0}^{+\infty} \left|\partial_r(rv_0)\right|^m\,dr+\int_0^{+\infty} |rv_1|^m\,dr\approx E_m(v_0,v_1),$$
\item \label{P5} \emph{Energy conservation}. $E_m(\vec{v})$ is independent of time. We call $E_m$ the $\LLL^m$-modified energy for equation \eqref{LW}.
 \item \label{P6} \emph{Exterior energy bound}. If $R>0$, the following holds for all $t\geq 0$ or for all $t\leq 0$:
 $$\int_{R}^{+\infty} |\partial_r(rv_0)|^m+|rv_1|^m\,dr\lesssim \int_{R+|t|}^{+\infty} |\partial_r(rv)|^m+|\partial_t(rv)|^m\,dr.$$
 \end{enumerate}
\end{proposition}
Property \eqref{P5} follows from direct computation, and the formula \eqref{eq_v}. Let us mention that the notation $E_m$ has a slightly different meaning in \cite{DuRoy15P}.
\begin{remark}
Note that:
\begin{equation}
\label{13}E_{2}(v(t))=\int_{\mathbb R^{3}}|\nabla v(t,x)|^{2}dx+\int_{\mathbb R^{3}}|\partial_{t} v(t,x)|^{2}dx\,,
\end{equation}
which coincides (up to a constant) with the standard energy functional for \eqref{1}. 
Moreover, from \eqref{2} we know for any $m\in (1,+\infty)$\,, there exists $C_{m}>0$ such that
\begin{equation}
\label{11}
C_{m}^{-1} \,\|\vec{v}(0)\|_{\LLL^m}\leq\, \|\vec{v}(t)\|_{\LLL^m}\leq\,C_m \|\vec{v}(0)\|_{\LLL^m},\quad\forall \; t\,.
\end{equation}
Thus $\|\vec{v}(t)\|_{\LLL^m}$ enjoys the pseudo-conservation law, namely \eqref{11}, and extends the 
classical energy to the general case $m>1$.
\end{remark}
 From the conservation of the energy, we deduce the following energy estimate for the equation with a right-hand side.
\begin{corollary}
Consider the problem
\begin{equation}
\label{26}
(\partial_{t}^{2} -\Delta) u(t,x)=f(t,x)\,,\,
(u,\partial_{t}u)|_{t=0}=(u_{0},u_{1})\,,\,t\in \mathbb R,\,x\in\mathbb R^{3}\,,
\end{equation} with $(u_{0}, u_{1})\in \LLL^m$ for a fixed $m>1$, and $f$ is radial. Then  we have the following inequality as long as the righthand side is finite,
\begin{equation}
\label{energy-est}
\begin{split}
\sup_{t\in\R}\Bigl(\int_{0}^{\infty}&\Bigl[|\partial_{r}(ru)|^{m}(t)+ |\partial_{t}(ru)|^{m}(t)\Bigr]\,dr\Bigr)^{\frac1m}\\
\leq &C\Bigl(\|(u_0,u_1)\|_{\LLL^m}+\int_{-\infty}^{+\infty}\Bigl(\int_{0}^{\infty}|rf(t,r)|^{m}dr\Bigr)^{\frac{1}{m}}dt\Bigr)
\end{split}
\end{equation}
\end{corollary}
\begin{proof}
Write $u(t,r)=u_{\rm L}(t,r)+u_{\mathcal N}(t,r)$
with 
\[u_{\rm L}(t,r)=S_{\rm L}(t)(u_{0},\,u_{1})\,,\quad
u_{\mathcal N}(t)=\int_{0}^{t}\frac{\sin(t-s)\sqrt{-\Delta}}{\sqrt{-\Delta}}f(s)\,ds\,,\]
The bound for $\|\vec{u}_{\rm L}\|_{\LLL^m}$ follows from \eqref{11} and the conservation of the $\LLL^m$ modified energy. Moreover,
$$\left\|\vec{u}_{\mathcal N}(t,r)\right\|_{\LLL^m}\leq \int_0^t\left\|\left( \frac{\sin\left((t-s)\sqrt{-\Delta}\right)}{\sqrt{-\Delta}}f(s),\cos\left((t-s)\sqrt{-\Delta}\right)f(s)\right)\right\|_{\LLL^m}\,ds,$$
and the estimate on $u_{\NNN}$ follows again from \eqref{11} and the conservation of the $\LLL^m$-modified energy. The proof is complete.
\end{proof}

\subsection{Strichartz estimates in weighted Sobolev spaces}
\label{SS:Strichartz}
Let $\Omega$ be a measurable subset of $\R_t\times (0,+\infty)$ of the form $\Omega=\bigcup_{t\in \R} \{t\}\times J_t$ where for all $t$, $J_t$ is a measurable subset of $(0,+\infty)$. If $f$ is a measurable function on $\Omega$, we let
$$\|f\|_{S(\Omega)}=\left(\int_{\R}\left(  \int_{J_t} |f(t,r)|^{(2m+1)m}r^mdr\right)^{\frac{1}{m}}dt\right)^{\frac{1}{2m+1}}.$$
If $\Omega=I\times (0,+\infty)$, where $I$ is a time interval, we will denote $S(\Omega)=S(I)$ to lightened notations:
$$\|f\|_{S(I)}=\left(\int_{I} \left(  \int_{0}^{+\infty} |f(t,r)|^{(2m+1)m}r^mdr\right)^{\frac{1}{m}}dt\right)^{\frac{1}{2m+1}}.$$
In this subsection we prove the following Strichartz estimate:
\begin{proposition}
\label{prop:key-estimate}
Let $m>1$ and assume $v(t,x)$ be the solution of the Cauchy problem \eqref{1} with radial initial data $(v_{0}, v_{1})\in \mathcal L^{m}$.
Then there exists a constant $C$ such that
\begin{equation}
\label{27}
\|v\|_{S(\RR)}
\leq C \| \vec{v}(0)\|_{\LLL^m}\,.
\end{equation}
\end{proposition}
We also have its analog for the inhomogeneous part:
\begin{proposition}
\label{prop:crucial_bis}
Let $m>1$ and $u(t,r)$ be the solution of \eqref{26} with $\vec{u}(0)=(0,0)$. Assume 
\[\|f\|_{L^{1}_{t}L^{m}_{x}(r^{m}dr)}:=\int_{-\infty}^{+\infty}\Bigl(\int_{0}^{+\infty}|f(t,r)|^{m}r^{m}dr\Bigr)^{\frac1m}dt<\infty\,.\]
Then we have 
\begin{equation}
\label{27-b}
\|u\|_{S(\RR)}\leq C \|f\|_{L^{1}_{t}L^{m}_{x}(r^{m}dr)}\,.
\end{equation}
\end{proposition}
We start by proving auxiliary symmetric Strichartz-type estimates
in \S \ref{SS:sym_Stri}, using the weak continuity in $L^1$ of the Hardy-Littlewood maximal function. In \S \ref{SS:key_Stri} we will interpolate these estimates with standard Strichartz inequalities to obtain the key estimates \eqref{27} and \eqref{27-b}.

\subsubsection{A family of symmetric Strichartz estimates}
\label{SS:sym_Stri}
With the explicit expression \eqref{2}, we are ready to deduce a crucial estimate
for the linear wave equation \eqref{1} with  $\vec{v}(0)\in \LLL^m$.
\begin{proposition}
\label{prop:linear-symmetric-estimate}
Let $v(t,x)=S_{\rm L}(t)(v_0,v_1)$ be a radial solution of \eqref{1}.
Then for any $m\in\,(1,+\infty)$ and $\alpha\in\,(1,+\infty)$\,, 
there is a constant $C$ such that the following \emph{a priori} estimate is valid 
\begin{equation}
\label{12}
\Bigl(\int_{\mathbb R}\int_{0}^{+\infty}
\bigl |v(t,r)\bigr|^{\alpha m}r^{\alpha-2}drdt\Bigr)^{\frac{1}{\alpha m}}\leq \,C\,\|\vec{v}(0)\|_{\LLL^m}.
\end{equation}
\end{proposition}

%

\begin{proof}
We assume $v_{1}\equiv 0$ first. Then from \eqref{3} and the fundamental theorem of calculus, 
\begin{equation}
\label{14}
v(t,r)=\frac{1}{2r}\int_{t-r}^{t+r}\partial_{s}\bigl(s\, v_{0}(|s|)\bigr)\,ds,\quad r=|x|\,.
\end{equation}
Let us consider the operator 
\begin{equation}
\label{T}\mathcal T:\, G(s)\mapsto \frac{1}{2r}\int_{t-r}^{t+r}G(s)\,ds.\end{equation}
First, it is clear that 
\begin{equation}
\label{15}
\sup_{(t,r)\in \mathbb R\times \mathbb{R}_{+}}
|\mathcal TG(t,r)|\leq\,\bigl\|G\bigr\|_{L^{\infty}(\mathbb R;\,ds)}.
\end{equation}
Next, we demonstrate the following weak type estimate
\begin{equation}
\label{16}
\bigl\|\mathcal T G\bigr\|_{L^{\alpha}_{w}(\mathbb R\times\mathbb R_{+}; \,r^{\alpha-2}drdt)}\leq C\|G\|_{L^{1}(\mathbb R;\,ds)}\,,
\end{equation}
or equivalently, there is $C>0$ such that for any $\lambda>0$ we have
\begin{equation}
\label{17}
\iint_{\mathcal E_{\lambda}}r^{\alpha-2}drdt\leq C\Bigl(\frac{\|G\|_{L^{1}}}{\lambda}\Bigr)^{\alpha},
\end{equation}
where $\mathcal E_{\lambda}=\{(t,r)\in \mathbb R\times\mathbb R_{+}: |\mathcal T G(t,r)|>\lambda\}$.
\vskip0.2cm

Given this, we have, interpolating between \eqref{15} and \eqref{16} ( see Theorem 5.3.2 in \cite{BeLo76B} )
\begin{equation}
\label{good-estimate}
\Bigl(\int_{\mathbb R}\int_{0}^{+\infty}|\mathcal T G(t,r)|^{\alpha m}r^{\alpha-2}drdt\Bigr)^{\frac{1}{\alpha}}\leq C\, \int_{\mathbb R}|G(s)|^{m}ds.
\end{equation}
The estimate \eqref{12} with $v_{1}\equiv0$ now follows by using \eqref{good-estimate} with 
\[G(s)=\partial_{s}\bigl(s\,v_{0}(|s|)\bigr).\]
\vskip0.2cm

To show \eqref{17}, one observes that on $\mathcal E_{\lambda}$, 
$$0<r<\frac{\|G\|_{L^{1}}}{\lambda}\quad
\text{and}\quad(\mathcal M G)(t)>\lambda\,, $$ 
where $\mathcal M$ denotes the Hardy-Littlewood maximal function. Therefore, we can bound
the left hand side of \eqref{17} as follows
\begin{align}
 \int^{\|G\|_{L^{1}}/2\lambda}_{0}r^{\alpha-2}dr\int_{\{t\in\R\mid (\mathcal M G)(t)>\lambda\}}dt\leq C\Bigl(\frac{\|G\|_{L^{1}}}{\lambda}\Bigr)^{\alpha},
\end{align}
where we have used the weak estimate $\mathcal M:\,L^{1}(\mathbb R)\rightarrow L^{1}_{w}(\mathbb R)$.
\vskip0.2cm

The case $v_{0}\equiv 0$ follows from the same argument. Indeed, in this case we have:
\begin{equation}
\label{19}
v(t,r)=\frac{1}{2r}\int_{t-r}^{t+r}sv_{1}(|s|)\,ds\,.
\end{equation}
Letting $G(s)=sv_{1}(|s|)$ and applying \eqref{good-estimate} we are done.
\end{proof}
Let $u(t,x)$ be a solution to the nonhomogeneous Cauchy problem (\ref{26}),
where $f(t,x)$ is radial in the space variable and locally integrable. 
If we set 
\begin{equation}
\label{6}
g(t,\rho)=\rho\,f(t,|\rho|),
\end{equation}
then we have
\begin{equation}
\label{5}
u(t,r)=\frac{1}{2r}\int_{0}^{t}\int_{\tau-r}^{\tau+r}g(t-\tau,\,\sigma)\,d\sigma d\tau.
\end{equation}
After a change of variables, we obtain 
\begin{equation}
\label{7}
u(t,r)=\frac{1}{2r}\int_{t-r}^{t+r}G(t,\rho)\,d\rho\,,
\end{equation}
with 
\[G(t,\rho)=\int_{0}^{t}g(s,\rho-s)\,ds\,.\]
A proof very close to the one of Proposition \ref{prop:linear-symmetric-estimate} yields symmetric Strichartz estimates for the nonhomogeneous equation:
\begin{proposition}
\label{symmetric-estimate-for-nonhmg}
Let $u(t,x)$ be a radial solution of the problem \eqref{26} with initial data $\vec{u}(0)=(0,0)$. Then for any $m\in \,(1,+\infty)$  and $\alpha\in \,(1,+\infty)$\, there is a constant $C$ such that we have 
\begin{equation}
\label{24}
\Bigl(\int_{\mathbb R}\int_{0}^{+\infty} |u(t,r)|^{\alpha m}r^{\alpha-2}drdt\Bigr)^{\frac{1}{\alpha m}}
\leq C\,\int_{\mathbb R}\Bigl(\int_{0}^{+\infty}|rf(t,r)|^{m}dr\Bigr)^{\frac1m}dt.
\end{equation}
\end{proposition}
\begin{proof}
In view of \eqref{7}, we have
\[|u(t,r)|\leq \mathcal T \widetilde G(t,r)\,,\]
where $\mathcal T$ is defined as in \eqref{T} and 
\[\widetilde G(\rho)=\int_{-\infty}^{+\infty}\bigl|g(s,\rho-s)\bigr|ds\,,\]
with $g$ given by \eqref{6}. 
Noting that $m>1$, we obtain \eqref{24} by using  \eqref{good-estimate} and Minkowski's inequality.
\end{proof}
\begin{remark}
Notice that from \eqref{14} and \eqref{19}, one may deduce the following end-point Strichartz estimate
for linear wave equations in three dimensions with radial
initial data
\begin{equation}
\label{radial-edpt-Strichartz}
\bigl\|S_{\rm L}(t)(v_{0},v_{1})\bigr\|_{L^{2}(\mathbb R_{t},\,L^{\infty}(\mathbb R^{3}_{x}))}
\leq C\,\bigl(\|v_{0}\|_{\dot H^{1}(\R^{3})}+\|v_{1}\|_{L^{2}(\R^{3})}\bigr)\,,
\end{equation}
where $(v_{0},v_{1})\in \dot H^{1}_{rad}(\mathbb R^{3})\times L^{2}_{rad}(\mathbb R^{3})$.
In fact, we may assume without loss of generality that 
$(v_{0}, v_{1})$ belongs to the Schwartz class. Then \eqref{radial-edpt-Strichartz} follows from \eqref{14}
and \eqref{19} by using the $L^{2}$-boundedness 
of the Hardy-Littlewood maximal function and integration by parts.
\end{remark}

\subsubsection{Proof of the key Strichartz inequality}
\label{SS:key_Stri}
We prove here Propositions \ref{prop:key-estimate} and \ref{prop:crucial_bis}. Let us first recall
the following classical Strichartz estimates for wave equations (see \cite{GiVe95}).
\begin{theorem}
 \label{theo:GV}
Consider $v(t,x)$ the solution of the linear Cauchy problem
\begin{equation}
\label{eq:linear ondes}
\begin{cases}
&(\partial_{t}^{2}-\Delta)v=h(t,x)\qquad\quad (x,t)\in\R^{3}\times \R\\
&v\big|_{t=0}=v_{0}\in \dot H^{1}(\R^{3})\\
&\partial_{t}v\big|_{t=0}=v_{1}\in L^{2}(\R^{3})
\end{cases}
\end{equation}
so that 
\[v(t)=S_{\rm L}(t)(v_{0},v_{1})+\int_{0}^{t}\frac{\sin(t-s)\sqrt{-\Delta}}{\sqrt{-\Delta}}h(s)\,ds\,.\]
Let $2\leq q,\sigma\leq \infty$ and let the following conditions be satisfied
\[\frac1q+\frac1\sigma\leq \frac12,\quad (q,\sigma)\neq (2,\infty),\quad \frac1q+\frac3\sigma=\frac12\,.\]
Then there exists $C>0$,  such that $v$ satisfies the estimate
\begin{equation}
\label{eq:GV}
\|v\|_{L^{q}(\mathbb R,\, L^{\sigma}(\R^{3}))}
\leq C
\bigl(\|v_{0}\|_{\dot H^{1}(\R^{3})}+\|v_{1}\|_{L^{2}(\R^{3})}+\|h\|_{L^{1}(\mathbb R\,;\,L^{2}(\mathbb R^{3})}\bigr).
\end{equation}
\end{theorem}

We are now ready to prove Proposition \ref{prop:key-estimate}
\begin{proof}
Since \eqref{27} is classical when $m=2$,
it suffices to consider below the cases for $m>2$ and $1<m<2$ separately.
\medskip

If $m>2$, we denote by $m^{*}=2m$ and take 
$\alpha=\frac{4}{3}(2m+1)$. Then we have from \eqref{12}
\begin{equation}
\label{inter-1}
\Bigl(\int_{-\infty}^{+\infty}\int_{0}^{+\infty}|v(t,r)\,r^{\gamma_{1}}|^{\alpha m^{*}}r^{\gamma_{2}}drdt\Bigr)^{\frac{1}{\alpha m^{*}}}
\leq C\,\|\vec{v}(0)\|_{\LLL^{m*}},
\end{equation}
where
\[\gamma_{1}=\frac{5m-2}{5m(2m+1)},\quad \gamma_{2}=\frac{2}{5},\]
so that $\gamma_1\alpha m^*+\gamma_2=\alpha-2$.
 Let
\[q=\frac{8m(2m+1)}{8m^{2}-11m+6},\; \sigma=\frac{8m(2m+1)}{5m-2}\,.\]
Then \eqref{eq:GV} yields
\begin{equation}
\label{inter-2}
\Bigl(\int_{-\infty}^{+\infty}\Bigl(\int_{0}^{+\infty}|v(t,r)\,r^{\gamma_{1}}|^{\sigma}r^{\gamma_{2}}dr\Bigr)^{\frac q\sigma}dt\Bigr)^{\frac1q}
\leq C \|\vec{v}(0)\|_{\LLL^2}\,.
\end{equation}
In view of 
\[\frac{1}{m}=\frac{\theta}{2}+\frac{1-\theta}{m^{*}}\,,\,
\frac{1}{2m+1}=\frac{\theta}{q}+\frac{1-\theta}{\alpha m^{*}}\,,\,
\frac{1}{m(2m+1)}=\frac{\theta}{\sigma}+\frac{1-\theta}{\alpha m^{*}}\,,\,\theta=\frac{1}{m-1}\,,\]
and the fact that $\gamma_{1}m(2m+1)+\gamma_{2}=m\,$,
we obtain \eqref{27} by interpolating \eqref{inter-1} and \eqref{inter-2} (see Theorem 5.1.2 in \cite{BeLo76B}). 
\medskip

If $1<m<2$, we set 
\[
m^{*}=\frac{m+1}{2}\,,\;
\alpha=\frac{8(2m+1)}{3m+5}\,,\;\theta=\frac{2(m-1)}{m(3-m)}\,,\;
q=\frac{8(2m+1)}{10-m}\,,\;\sigma=\frac{8(2m+1)}{3m-2}\,,\]
\[
\gamma_{1}=\frac{3m-2}{6m^{2}+11m+4}=\frac{3m-2}{(2m+1)(3m+4)}\,,\quad\gamma_{2}=\frac{6m}{3m+4}\,.
\]
One can verify that \eqref{inter-1} and \eqref{inter-2} along with the interpolation relations as in the first case remain valid.
This completes the proof.
\end{proof}

Using the same argument as above and \eqref{24}, we obtain Proposition \ref{prop:crucial_bis}.
\medskip

We conclude this subsection by some additional Strichartz-type estimates that will be useful in the construction of the profile decomposition in Section \ref{S:profiles} and follow from Proposition \ref{prop:linear-symmetric-estimate} and 
\eqref{radial-edpt-Strichartz}.
\begin{proposition}
\label{prop:subkey-estimate}
Assume $m>2$ and $v(t,x)$ is the solution of the Cauchy problem \eqref{1} with radial initial data $(v_{0},v_{1})\in\mathcal L^{m}$. Let 
\[a=\frac{2m(m-1)(m+2)}{m^{2}+3m-2},\;
b=\frac{2m(m-1)(m+2)}{m-2}\,.\]
Then there exists a constant $C$ such that 
\begin{equation}
\label{subkey-estimate}
\Bigl(\int_{-\infty}^{+\infty}\Bigr(\int_{0}^{+\infty}|v(t,r)|^{b}r^{m}dr\Bigr)^{\frac ab}dt\Bigr)^{\frac{1}{a}}\leq C\, \|\vec{v}(0)\|_{\LLL^m}
\end{equation}
\end{proposition}
\begin{proof}
Indeed, from \eqref{12}, we have 
\begin{equation}
\label{E1}
\Bigl(\int_{-\infty}^{+\infty}\int_{0}^{+\infty}|v(t,r)|^{2m(m+2)}r^{m}drdt\Bigr)^{\frac{1}{2m(m+2)}}\leq C\,\|\vec{v}(0)\|_{\LLL^{2m}}\,.
\end{equation}
Interpolating \eqref{E1} with \eqref{radial-edpt-Strichartz}, we are done.
\end{proof}
The choice of $(a,b)$ above is not suitable in the case $m<2$, where we will use the following estimates:
\begin{proposition}
\label{prop:subkey-estimate-1}
Assume $1<m<2$ and $v(t,x)$ is the solution of the Cauchy problem \eqref{1} with radial initial data $(v_{0},v_{1})\in\mathcal L^{m}$. Let 
\[a=\frac{m(m+2)(3-m)}{m^{2}-m+2},\;
b=\frac{m(m+2)(3-m)}{2(2-m)}\,.\]
Then there exists a constant $C$ such that 
\begin{equation}
\label{subkey-estimate-1}
\Bigl(\int_{-\infty}^{+\infty}\Bigr(\int_{0}^{+\infty}|v(t,r)|^{b}r^{m}dr\Bigr)^{\frac ab}dt\Bigr)^{\frac{1}{a}}\leq C\, \|\vec{v}(0)\|_{\LLL^m}
\end{equation}
\end{proposition}
\begin{proof}
Let $m^{*}=(m+1)/2$.  From \eqref{12}, we have 
\begin{equation}
\label{E1-1}
\Bigl(\int_{-\infty}^{+\infty}\int_{0}^{+\infty}|v(t,r)|^{m^{*}(m+2)}r^{m}drdt\Bigr)^{\frac{1}{(m+2)m^*}}\leq C\,\|\vec{v}(0)\|_{\LLL^{m^*}}\,.
\end{equation}
Interpolating \eqref{E1-1} with \eqref{radial-edpt-Strichartz}, we are done.
\end{proof}
\begin{remark}
In both propositions, we have:
$m<a<2m+1<b/m<\infty$.
\end{remark}

\begin{remark}	
The interpolations we used in the above two propositions are based on the complex method.
In fact, we used Theorem 5.1.1 and Theorem 5.1.2 in \cite{BeLo76B}. 
\end{remark}
\begin{remark}
Notice that when $m=2$, $(a,b)=(2,\infty)$ coincides with the end-point Strichartz estimate \eqref{radial-edpt-Strichartz}.
\end{remark}

\subsection{Local well-posedness}

Consider here the Cauchy problem for the nonlinear wave
equations \eqref{NLW}, \eqref{ID}, with $(u_0,u_1)\in \LLL^m$, $m>1$. 
In this subsection, we prove the following small-data well-posedness statement, which implies Theorem \ref{T:LWP}:
\begin{proposition}
\label{prop:LWP}
There exists $\delta_0>0$ such that if $0\in I\subset \mathbb R$ is an interval and 
\begin{equation}
\label{local-small}
\bigl\|S_{\rm L}(t)(u_{0},u_{1})\bigr\|_{S(I)}
=\delta\leq \delta_0\,,
\end{equation}
then there exists a unique solution $u\in S(I)$ such that 
$\vec{u}\in C^0(I,\LLL^m)$ to the Cauchy problem \eqref{NLW}, \eqref{ID} for $t\in I$. Moreover:
\begin{equation}
\label{solution-est-1}
\|u\|_{S(I)}\leq 2\delta
\end{equation}
and
we have
\begin{equation}
\label{solution-est-2}
\sup_{t\in I}\|\vec{v}(t)\|_{\LLL^m}\leq C_{m}(\|(u_{0},u_{1})\|_{\mathcal L^{m}}+\delta^{2m+1})\,.
\end{equation}
\end{proposition}
\begin{remark}
From the assumption on the initial data and Strichartz type inequality \eqref{27}, we see that for each $(u_0,u_1)\in \LLL^m$ and $\delta>0$, there is an interval $I=I(u_0,u_1,\delta)$ such that \eqref{local-small} holds. Using this observation and standard arguments, it is easy to construct  from Proposition \ref{prop:LWP} a maximal solution of \eqref{NLW}, \eqref{ID} that satisfies the conclusion of Theorem \ref{T:LWP}.
\end{remark}

\begin{proof}
Let $C_{0}$ be the constant in the estimates \eqref{27} and \eqref{27-b} .
Consider 
\[\mathfrak X=\{v\;\text{on} \; \mathbb R\times \mathbb R^{3}\mid v(t,x)=v(t,|x|)\,,\,
\|v\|_{S(I)}\leq 2\delta\}\,,\]
where
\[0<\delta<\min(C_{0}^{-\frac{1}{p-1}}2^{-\frac{p}{p-1}},2^{-\frac{p+2}{p-1}}(pC_{0})^{-\frac{1}{p-1}}),\,
p=2m+1\,.\]
Define
\begin{equation}
\label{35}
\Phi_{(u_{0},u_{1})}(v)=S_{\rm L}(t)(u_{0},u_{1})
+\iota\int_{0}^{t}\frac{\sin (t-s)\sqrt{-\Delta}}{\sqrt{-\Delta}}|v|^{2m}v(s)\,ds\,.
\end{equation}
If $v,w \in \mathfrak X$, we have from \eqref{27-b}
\begin{equation*}
\label{36}
\bigl\|\Phi_{(u_{0},u_{1})}(v)\bigr\|_{S(I)}
\leq \delta+C_{0}(2\delta)^{p}\leq 2\delta\,,
\end{equation*}
and by H\"older inequality
\begin{align*}
\bigl\|\Phi_{(u_{0},u_{1})}(v)-\Phi_{(u_{0},v_{0})}(w)
\bigr\|_{S(I)}
\leq&\,2\,p\,C_{0}(\|v\|^{p-1}_{S(I)}+\|w\|^{p-1}_{S(I)})
\|v-w\|_{S(I)}\\
\leq &\,4\,p\,C_{0}(2\delta)^{p-1}
\|v-w\|_{S(I)}\\
\leq&\,\frac12\, \|v-w\|_{S(I)}\,,
\end{align*}
for all $v, \, w\in \mathfrak X$ . Thus, there exists a unique fixed point $u\in \mathfrak X$ such that 
$$u=\Phi_{u_{0}, u_{1}}(u)\,.$$
Notice that \eqref{solution-est-1} follows from the construction and \eqref{solution-est-2} follows from the energy estimates and \eqref{solution-est-1}.
\end{proof}
\subsection{Exterior long-time perturbation theory}
We conclude this section by a long-time perturbation theory result for equation \eqref{NLW} with initial data in $\LLL^m$. Taking into account the finite speed of propagation, we will give a statement that works as well when the estimates are restricted to the exterior $\{r>A+|t|\}$ of a wave cone. This generalization will be very useful when using the channels of energy arguments.
\begin{lemma}
 \label{L:ELTPT}
 Let $M>0$. There exists $\eps_M>0$, $C_M>0$ with the following properties.
 Let $T\in (0,+\infty]$, $u,\tilde{u}\in S((0,T))$ such that $\vec{u}, \vec{\tilde{u}}\in C^0([0,T),\LLL^m)$. Assume that $u$ is a solution of \eqref{NLW}, \eqref{ID} on $[0,T)$ and that\footnote{in the sense that $\tilde{u}$ satisfies the usual integral equation}
 \begin{equation}
 \label{eq_tildeu}
  \left\{\begin{aligned}
\partial_t^2\tilde{u}-\Delta\tilde{u}=\iota\indic_{\{r\geq (A+|t|)_+\}} |\tilde{u}|^{2m}\tilde{u}+e\\
\tilde{u}_{\restriction t=0}=(\tilde{u}_0,\tilde{u}_1),
         \end{aligned}\right.
 \end{equation} 
 where $e\in L^{1}_tL^m_x(r^m\,dr)$, $A\in \RR\cup\{-\infty\}$. Let 
 $$R_{\rm L}(t)=S_{\rm L}(t)\left( (u_0,u_1)-(\tilde{u}_0,\tilde{u}_1) \right).$$
 Assume
 \begin{gather}
\label{LTPT_bound_M}
 \left\|\tilde{u}\right\|_{S\left(\{t\in (0,T),\; r\geq (A+|t|)_+\}\right)}\leq M\\
 \label{LTPT_small}
 \int_0^T\left( \int_{(A+|t|)_+}^{+\infty} |r\,e|^m\,dr\right)^{\frac 1m}\,dt+\left\|R_{\rm L}\right\|_{S(\left\{t\in [0,T),\; r\geq (A+|t|)_+\right\})}=\eps\leq \eps_M.
 \end{gather}
 Then $u(t)=\tilde{u}(t)+R_{\rm L}(t)+\epsilon(t)$ with 
 $$ \left\|\epsilon\right\|_{S\left( \{t\in[0,T),\; r\geq (A+|t|)_+\} \right)}+\sup_{t\in [0,T)} \int_{(A+|t|)_+} |r\partial_{t,r}\epsilon|^m\,dr\leq C_M\eps.$$
\end{lemma}
In the lemma, we have denoted by $(A+|t|)_+=\max(0,A+|t|)$. By convention, if $A=-\infty$, this quantity equals $0$ for all $t$. Note that the case $A=-\infty$ corresponds to the usual long-time perturbation theory statement (see e.g. \cite{TaVi05})\footnote{traditionally the ``linear part'' of the solution $R_{\rm L}(t)$ is incorporated into $\tilde{u}$. For convenience we preferred to distinguish between these two components}.
\begin{proof}[Sketch of the proof]
 We let, for $t\in [0,T)$.
 \begin{align*}
 \mfkE(t)&=\left( \int_{(A+|t|)_+}^{+\infty}\left|\epsilon(t,r)\right|^{(2m+1)m}r^m\,dr \right)^{\frac{1}{(2m+1)m}}\\
\mfktU(t)&=\left( \int_{(A+|t|)_+}^{+\infty}\left|\tilde{u}(t,r)\right|^{(2m+1)m}r^m\,dr \right)^{\frac{1}{(2m+1)m}}\\
\mfkR(t)&=\left( \int_{(A+|t|)_+}^{+\infty}\left|R_{\rm L}(t,r)\right|^{(2m+1)m}r^m\,dr \right)^{\frac{1}{(2m+1)m}}.
 \end{align*}
 By the assumptions \eqref{LTPT_bound_M}, \eqref{LTPT_small}, 
 $$\|\mfktU\|_{L^{2m+1}(0,T)}\leq M,\quad \|\mfkR\|_{L^{2m+1}(0,T)}\leq \eps.$$
 Since 
 $$(\partial_t^2-\Delta)\epsilon=\iota\left(|u|^{2m}u-|\tilde{u}|^{2m}\tilde{u}\right)+e,$$
 we obtain by \eqref{energy-est}, Strichartz estimates and finite speed of propagation that for all $\theta\in [0,T)$,
 \begin{multline}
 \label{Str_LTPT}
 \sup_{t\in [0,\theta]}\left[\left(\int_{(A+|t|)_+}^{+\infty}\left|r\partial_{t,r}\epsilon\right|^m\,dr\right)^\frac{1}{m}+
\|\vec{\epsilon}(t)\|_{\LLL^m}+\|\mfkE(t)\|_{L^{2m+1}}\right]\\
 \leq C\int_0^{\theta}\left(\int_{(A+|t|)_+}^{+\infty}\left(\left||\tilde{u}|^{2m}\tilde{u}-|u|^{2m}u\right|^m+|e|^m\right)r^mdr\right)^{\frac 1m}\,dt.
 \end{multline}
 We have 
 $$\int_0^{\theta}\left(\int_{(A+|t|)_+}^{+\infty}|e|^mr^mdr\right)^{\frac 1m}\,dt \leq \eps$$
 and, using H\"older's inequality 
 \begin{multline*}
\int_0^{\theta}\left(\int_{(A+|t|)_+}^{+\infty}\left||\tilde{u}|^{2m}\tilde{u}-|u|^{2m}u\right|^mr^mdr\right)^{\frac 1m}\,dt 
 \\ \lesssim \int_0^{\theta}\left( \mfkE(t)+\mfkR(t) \right)\left( \mfktU(t)^{2m}+\mfkR(t)^{2m}+\mfkE(t)^{2m} \right)\,dt\\
 \leq C\left(\|\mfkE\|_{L^{2m+1}(0,\theta)}^{2m+1}+\|\mfkR\|_{L^{2m+1}(0,\theta)}^{2m+1}+\int_0^{\theta}\mfkR(t)\mfktU(t)^{2m}\,dt+\int_0^{\theta}\mfkE(t)\mfktU(t)^{2m}\,dt\right)\\
 \leq C\left(\|\mfkE\|_{L^{2m+1}(0,\theta)}^{2m+1}+\eps^{2m+1}+M^{2m}\eps+\int_0^{\theta}\mfkE(t)\mfktU(t)^{2m}\,dt\right).
 \end{multline*}
Collecting the above, we obtain, for all $\theta\in [0,T)$,
\begin{equation*}
\|\mfkE\|_{L^{2m+1}(0,\theta)}\leq  C\left(\eps+\eps^{2m+1}+M^{2m}\eps+\|\mfkE\|_{L^{2m+1}(0,\theta)}^{2m+1}+\int_0^{\theta}\mfkE(t)\mfktU(t)^{2m}\,dt\right)
\end{equation*}
This is a Gr\"onwall-type inequality classical in this context. Using e.g. Lemma 8.1 in \cite{FaXiCa11}, we deduce that for all $\theta\in [0,T)$,
\begin{equation*}
 \|\mfkE\|_{L^{2m+1}(0,\theta)}\leq C\left(\eps+\eps^{2m+1}+M^{2m}\eps+\|\mfkE\|_{L^{2m+1}(0,\theta)}^{2m+1}\right)\Phi(CM^{2m}),
\end{equation*}
where $\Phi(s)=2\Gamma(3+2s)$, and $\Gamma$ is the usual Gamma function. Using a standard bootstrap argument, we deduce, assuming that $\eps\leq \eps_M$ for some small $\eps_M$
$$\|\mfkE\|_{L^{2m+1}(0,\theta)}\leq C_M\eps,$$
and going back to \eqref{Str_LTPT} and the computations that follow this inequality, we obtain also the desired bound on the $\LLL^m$ norm of $\epsilon$. The proof of the lemma is complete.
\end{proof}

\section{Profile decomposition}
\label{S:profiles}
\subsection{Linear profile decomposition}
The main result of this section is the following:
\begin{theorem}
\label{T:profiles}
 Let $(u_{{\rm L},n})_n$ be a sequence of radial solutions of \eqref{LW} such that $(\vec{u}_{{\rm L},n}(0))_n$ is bounded in $\LLL^m$. Then there exists a subsequence of $(u_{{\rm L},n})_n$ {\rm (}still denoted by $(u_{{\rm L},n})_n${\rm)} and, for all $j\geq 1$, a solution $U^j_{\rm L}$ of \eqref{LW} with initial data $(U_0^j,U_1^j)$ in $\LLL^m$ and sequences $(\lambda_{j,n})_n\in(0,\infty)^{\NN}$, $(t_{j,n})_n\in \RR^{\NN}$ such that the following properties hold.
 \begin{itemize}
  \item Pseudo-orthogonality: for all $j,k\geq 1$, one has 
  \begin{equation}
   \label{P-O} j\neq k\Longrightarrow \lim_{n\to\infty}\frac{\lambda_{j,n}}{\lambda_{k,n}}+\frac{\lambda_{k,n}}{\lambda_{j,n}}+\frac{|t_{j,n}-t_{k,n}|}{\lambda_{j,n}}=+\infty,
  \end{equation} 
  \item Weak convergence: for all $j\geq 1$,
  \begin{equation}
   \label{WeakCV}
  \left( \lambda_{j,n}^{\frac{1}{m}} u_{{\rm L},n}(t_{j,n},\lambda_{j,n}\cdot), \lambda_{j,n}^{\frac{1}{m}+1} \partial_t u_{{\rm L},n}(t_{j,n},\lambda_{j,n}\cdot)\right)\xrightharpoonup[n\to\infty]{}\left( U_0^j,U_1^j \right),
  \end{equation} 
   weakly in $\LLL^m$.
  \item Bessel-type inequality: for all $J\geq 1$,
  \begin{equation}
   \label{Pythagorean}
\lim_{n\to\infty} E_m(u_{0,n},u_{1,n})-\sum_{j=1}^J E_m\left( \vec{U}^j_{\rm L}(0) \right)\geq 0,
  \end{equation} 
 \item Vanishing in the dispersive norm:
 \begin{equation}
  \label{dispersive}
  \lim_{J\to\infty} \lim_{n\to\infty} \|w_n^J\|_{S(\RR)}=0,
 \end{equation} 
\end{itemize}
 where 
 \begin{gather}
\label{wnJ}
 w_n^J(t,x)=u_{{\rm L},n}(t,x)-\sum_{j=1}^J U_{{\rm L},n}^j(t,x),\\
\label{ULnJ}
 U_{{\rm L},n}^j(t,x)=\frac{1}{\lambda_{j,n}^{\frac{1}{m}}}U^j_{\rm L}\left( \frac{t-t_{j,n}}{\lambda_{j,n}},\frac{x}{\lambda_{j,n}} \right).
 \end{gather}
\end{theorem}

Theorem \ref{T:profiles} generalizes (in the radial setting) the profile decomposition of Bahouri and G\'erard \cite{BaGe99} to sequences that are bounded in $\LLL^m$ instead of the classical energy space. The only difference between the two decomposition is the fact that the Pythagorean expansion proved in \cite{BaGe99} is replaced by the weaker property \eqref{Pythagorean}. One cannot hope, in this context, to have an exact Pythagorean expansion: see the example p.387 of \cite{Jaffard99}.
\medskip

The proof of Theorem \ref{T:profiles} is based on the following two propositions that we will prove in Subsection \ref{SS:0Strichartz} and \ref{SS:Bessel} respectively.

\begin{proposition}
\label{P:wLn}
Let $(u_{{\rm L},n})_{n}$ be a sequence of radial solutions to
the linear wave equation and denote by $(u_{0,n}, u_{1,n})=\vec{u}_{{\rm L},n}(0)$.
Assume for $m\in(1,+\infty)$,  the sequence $\bigl(\vec{u}_{{\rm L},n}(0)\big)_{n}$ is bounded in $\LLL^{m}$ and that for all sequences $(\lambda_{n})_{n}\in (0,\infty)^{\mathbb N}$
and $(t_{n})_{n}\in\mathbb R^{\mathbb N}$,
\begin{equation}
\label{asmp}
\Biggl(\frac{1}{\lambda_{n}^{1/m}}u_{{\rm L},n}\Bigl(\frac{-t_{n}}{\lambda_{n}},\frac{\cdot}{\lambda_{n}}\Bigr),
\frac{1}{\lambda_{n}^{1+1/m}}\partial_{t}u_{{\rm L},n}\Bigl(\frac{-t_{n}}{\lambda_{n}},\frac{\cdot}{\lambda_{n}}\Bigr)\Biggr)_{n}\,,
\end{equation}
converges weakly to $0$ in $\LLL^{m}$ as $n\rightarrow+\infty$.
Then 
\begin{equation}
\label{str-norm-cvg}
\lim_{n\rightarrow+\infty}\|u_{{\rm L},n}\|_{S(\RR)}=0\,,
\end{equation}
\end{proposition}
\begin{proposition}
 \label{P:Pythagorean}
 Let $J\geq 1$ and $(U^j_{\rm L})_{j=1\ldots J}$ be solutions of the linear wave equations with initial data in $\LLL^m$. For all $j=1,\ldots, J$, we let $(\lambda_{j,n})_n\in (0,\infty)^{\NN}$ and $(t_{j,n})_n\in \RR^{\NN}$ be sequences of parameters that satisfy the pseudo-orthogonality property \eqref{P-O}. Let $(u_{{\rm L},n})$ be a sequence of solutions of the linear wave equation with initial data in $\LLL^m$. Let $w_n^J$ be defined by \eqref{wnJ}, \eqref{ULnJ}
 and assume that for all $j\in \{1,\ldots, J\}$,
 \begin{equation}
  \label{H:Pythagorean}
 \left( \lambda_{j,n}^{\frac 1m}w_{n}^J\left(t_{j,n},\lambda_{j,n}\cdot \right),\lambda_{j,n}^{\frac 1m+1}\partial_t w_{n}^J\left(t_{j,n},\lambda_{j,n}\cdot \right)\right) \xrightharpoonup[n\to\infty]{}0\text{ weakly in }\LLL^m. 
 \end{equation} 
 Then the Bessel-type inequality \eqref{Pythagorean} holds.
\end{proposition}

\begin{proof}[Proof of the theorem]
The proof of Theorem \ref{T:profiles} assuming Proposition \ref{P:wLn} and \ref{P:Pythagorean} is quite standard, at least in the Hilbertian setting. We give it for the sake of completeness. We mainly need to check that it is harmless
that we have only a Bessel-type inequality \eqref{Pythagorean}  in the $\LLL^m$ setting, which is not Hilbertian, instead of a more precise Pythagorean expansion.\medskip

We construct the profiles $U^j_{\rm L}$ and the parameters $\lambda_{j,n}$, $t_{j,n}$ by induction.

Let $J\geq 1$ and assume that for $1\leq j\leq J-1$, we have constructed profiles $U^j_{\rm L}$ such that \eqref{P-O} and \eqref{WeakCV} holds after extraction of a subsequence in $n$ (if $J=1$ we do not assume anything and set $w_n^0=u_{{\rm L},n}$). Note that it implies \eqref{Pythagorean} by Proposition \ref{P:Pythagorean}.  Let $\AAA_J$ be the set of $(U_0,U_1)\in \LLL^m$ such that there exist sequences $(\lambda_n)_n$, $(t_n)_n$ of parameters such that, after extraction of a subsequence:
$$ \left( \lambda_n^{\frac 1m}w_{n}^{J-1}(t_n,\lambda_n\cdot), \lambda_n^{\frac{1}{m}+1}\partial_tw_{n}^{J-1}(t_n,\lambda_n\cdot) \right)\xrightharpoonup[n\to\infty]{}(U_0,U_1),$$
weakly in $\LLL^m$, where $w_n^{J-1}$ is defined by \eqref{wnJ}.
We distinguish two cases.

\EMPH{Case 1} $\AAA_J=\Big\{(0,0)\Big\}$. In this case we stop the process and let $U_{\rm L}^j=0$ for all $j\geq J$. 

\EMPH{Case 2} There exists a nonzero element in $\AAA_J$. In this case, we choose $(U_0^J,U_1^J)\in \AAA_J$ such that 
\begin{equation}
 \label{choice_UJ}
 E_m\left( U_0^J,U_1^J \right)\geq \frac 12 \sup_{(U_0,U_1)\in \AAA_J} E_m(U_0,U_1),
\end{equation} 
and we choose sequences $(\lambda_{J,n})_n$ and $(t_{J,n})_n$ such that, (after extraction of subsequences in $n$),
\begin{equation}
 \label{def_param}
  \left( \lambda_{J,n}^{\frac{1}{m}} w_{n}^{J-1}(t_{J,n},\lambda_{J,n}\cdot), \lambda_{J,n}^{\frac{1}{m}+1} \partial_t w_{n}^{J-1}(t_{J,n},\lambda_{J,n}\cdot)\right)\xrightharpoonup[n\to\infty]{}\left( U_0^J,U_1^J \right)
\end{equation}
weakly in $\LLL^m$.
Note that  \eqref{WeakCV} holds for $j=J$ thanks to \eqref{def_param}. Furthermore \eqref{P-O} for $j\in \{1,\ldots J-1\}$, $k=J$ follows from \eqref{WeakCV} (for $j\in \{1,\ldots, J-1\}$), \eqref{def_param} and the fact that $(U_0^J,U_1^J)\neq (0,0)$. Finally, as already observed, \eqref{Pythagorean} is a consequence of \eqref{P-O}, \eqref{WeakCV} and Proposition \ref{P:Pythagorean}.\medskip

If there exists a $J\geq 1$ such that Case 1 above holds, then we are done: indeed, in this case, $w_n^J$ does not depend on $J$ for large $n$, and \eqref{dispersive} is an immediate consequence of the definition of $\AAA_J$ and Proposition \ref{P:wLn}. 

Next assume that Case 2 holds for all $J\geq 1$. Using a diagonal extraction argument, we obtain, for all $j\geq 1$,  profiles $U^j_{\rm L}$, and sequences of parameters $(\lambda^j_n)_n$ and $(t_n^j)_n$ such that \eqref{P-O}, \eqref{WeakCV} and \eqref{Pythagorean} hold for all $j,k,J$. It remains to prove \eqref{dispersive}. In view of Proposition \ref{P:wLn}, it is sufficient to prove:
\begin{equation*}
 \lim_{J\to\infty} \sup_{(A_0,A_1)\in \AAA_J} \left\|(A_0,A_1)\right\|_{\LLL^m}=0.
\end{equation*} 
This follows from \eqref{choice_UJ}, the equivalence between $E_m^{1/m}$ and the $\LLL^m$ norm, and the fact that, by \eqref{Pythagorean}, 
$$\lim_{J\to\infty} E_m\left( U_0^J,U_1^J \right)=0.$$
The proof is complete.
\end{proof}
\subsection{Convergence to $0$ of the Strichartz norm}
\label{SS:0Strichartz}
First of all, let us introduce the notation $\dot B^{s}_{\infty,\infty}(\mathbb R^{d})$ for the homogeneous Besov space on $\mathbb R^{d}$, which is defined as follows.
Let $\psi\in C^{\infty}_{0}(\mathbb R^{d})$ be a radial function, supported in $\{\xi\in\mathbb R^{d}: 1/2\leq |\xi|\leq 2\}$ and such that 
\[\sum_{j\in \mathbb Z}\psi(2^{-j}\xi)=1\,,\;\xi\in\mathbb R^{3}\setminus\{0\}\,.\]
We denote by $\dot \Delta_{j}$ the Littlewood-Paley projector 
\[\dot \Delta_{j}f(x)=\Bigl(\psi(2^{-j}\cdot)\widehat{f}(\cdot)\Bigr)^{\vee}(x)\;,\;j\in\mathbb Z\,,\]
where
\[\widehat{f}(\xi)=\int_{\mathbb R^{d}}f(x)e^{-ix\cdot\xi}dx\,,\]
is the Fourier transform on $\mathbb R^{d}$
and we use
\[g^{\vee}(x)=\frac{1}{(2\pi)^{d}}\int_{\mathbb R^{d}}
g(\xi)e^{ix\cdot\xi}d\xi\,,\]
to denote the inverse Fourier transform.
For a tempered distribution $f$ on $\mathbb{R}^d$, we set 
\[\|f\|_{\dot B^{s}_{\infty,\infty}(\mathbb R^{d})}:=
\sup_{j\in\mathbb Z}2^{js}\bigl\|\dot\Delta_{j}f\bigr\|_{L^{\infty}(\mathbb R^{d})}\;.\]
If $\|f\|_{\dot B^{s}_{\infty,\infty}}<+\infty$, we say  $f$
belongs to $ \dot B^{s}_{\infty,\infty}$.\medskip

We have the following refined Sobolev inequality in weighted norms.
\begin{lemma}
\label{lem:refined-Soblev-w}
Let $\omega(x)\in A_{p}$ with $1<p<+\infty$, {\emph i.e.}
\begin{equation}
\label{eq:Ap}
\sup_{B}\Bigl(\frac{1}{|B|}\int_{B}\omega(x)\,dx\Bigr)
\Bigl(\frac{1}{|B|}\int_{B}\omega(x)^{-\frac{1}{p-1}}\,dx\Bigr)^{p-1}<+\infty\,
\end{equation}
where the supremum is taken over all balls  $B$ in $\mathbb{R}^{d}$.
If $\nabla f\in L^{p}(\mathbb R^{d},\omega(x)dx)$ and $f\in \dot B^{-\beta}_{\infty,\infty}(\mathbb R^{d})$,
then
\begin{equation}
\label{eq:refined-Sobolev-w}
\|f\|_{L^{q}(\mathbb R^{d},\omega)}\leq C
\|\nabla f\|^{\theta}_{L^{p}(\mathbb R^{d},\omega)}
\|f\|_{\dot B^{-\beta}_{\infty,\infty}(\mathbb R^{d})}^{1-\theta}\,,
\end{equation}
where $1<p<q<+\infty$, $\theta=p/q$, $\beta=\theta/(1-\theta)$.
\end{lemma}

The refined Sobolev inequality \eqref{eq:refined-Sobolev-w} 
in weighted norms
was proved in \cite{Chamorro11},
where the author considered more general situations with the underlying 
domain $\mathbb R^{d}$
replaced by stratified Lie groups. 
The above lemma follows immediately since 
Euclidean spaces $\mathbb R^{d}$ with its natural group structure
is an example of a stratified Lie group.
Notice that $1\in A_{p}$, and one recovers the classical result 
on the refined Sobolev inequalities 
established first in  \cite{MeyerGerardOru97}.
\medskip
%

With Lemma \ref{lem:refined-Soblev-w} at hand, we are ready 
to prove the Proposition \ref{P:wLn}.

\begin{proof}
Since $\big((u_{0,n},u_{1,n})\big)_n$ is bounded in $\LLL^m$, there exists $A>0$ such that 
\[\int_0^{+\infty}|r\partial_{r} u_{0,n}(r)|^{m}dr
+\int_{0}^{+\infty}|ru_{1,n}(r)|^{m}dr\leq A<+\infty\,,\]
for all $n$.

Assuming \eqref{str-norm-cvg} fails, 
we have for some constant $c_{0}$ having the property that $0<c_{0}\leq C\,A^{1/m}$ such that
\begin{equation}
\label{limsup}
\limsup_{n\rightarrow\infty}\|u_{{\rm L},n}\|_{L^{2m+1}_{t}\left(\mathbb{R},\,L^{m(2m+1)}_{x}(\mathbb{R}^3,r^{m-2})\right)}=c_{0}\,,
\end{equation} 
where $C$ is the constant in \eqref{27}, \eqref{subkey-estimate} and  \eqref{subkey-estimate-1}.
From \eqref{subkey-estimate},  \eqref{subkey-estimate-1}  and H\"older's inequality, 
we know that up to a subsequence, there exists some $\theta\in(0 ,\,1)$ such that
\begin{equation}
\label{end-point-power-bd}
\lim_{n\rightarrow \infty}
\|u_{{\rm L},n}\|_{L^{\infty}_{t}\left(\mathbb{R},\,L^{m(m+1)}_{x}(\mathbb{ R}^3,r^{m-2})\right)}
\geq \biggl(\frac{c_{0}}{(CA^{1/m})^{\theta}}\biggr)^{\frac{1}{1-\theta}}\,.
\end{equation}
For $m>1$, we denote by $[m]$ the greatest integer lesser or equal to $m$
and by $\{ m\}:=m-[m]$ the fractional part of $m$. Notice that $\{m\}\in[0,1)$
and $\{m\}=0$ if and only if $m\in\mathbb N$.\medskip

Let $d=[m]+1$ and $\omega(x)=|x|^{\gamma}$ with $\gamma=\{m\}$, $x\in\mathbb{R}^d$.
It is easy to see that $\omega\in A_{m}$ (see for example \cite{Grafakos14B}) and we have the following 
refined Sobolev inequality in view of Lemma \ref{lem:refined-Soblev-w}
%
\begin{equation}
\label{refined-sobolev}
\|f\|_{L^{m(m+1)}(\mathbb R^{d},|x|^{\gamma})}
\leq C_{0}
\|\nabla f\|^{\frac{1}{m+1}}_{L^{m}(\mathbb R^{d},|x|^{\gamma})}
\| f\|^{\frac{m}{m+1}}_{\dot B^{-\frac{1}{m}}_{\infty,\infty}(\mathbb R^{d})}\,.
\end{equation}
If we apply \eqref{refined-sobolev}
to functions $u_{{\rm L},n}(t,|x|)$ with respect to the spatial 
variable $x\in \mathbb R^{[m]+1}$, we obtain
by transferring the formula into polar coordinates
\begin{equation}
\label{extension-to-higher-dim}
\begin{split}
\int_{0}^{+\infty} \bigl|u_{{\rm L},n}(t,r)\bigr|^{m(m+1)}r^{m}dr
&\leq C_{0}^{m(m+1)}\int_{0}^{+\infty} \bigl|r\partial_{r}u_{{\rm L},n}(t,r)\bigr|^{m}dr\\
\times\sup_{j\in\mathbb Z}\sup_{x\in\mathbb R^{[m]+1}}\Bigl(2^{-j/m}&
\int_{\mathbb R^{[m]+1}} \psi^{\vee}(y)\,u_{{\rm L},n}(t,|x-2^{-j}y|)\,dy\Bigr)^{m^{2}}\,.
\end{split}
\end{equation}
In view of the conservation of the $\LLL^m$-energy, and the fact that the norms $\|\cdot\|_{\LLL^m}$ and $(E_m)^{\frac 1m}$ are equivalent, there exists some $N>0$ such that if $n\geq N$
\begin{equation}
\label{Lower-bound-besov}
\sup_{t\in\mathbb R}\;\sup_{j\in\mathbb Z}\,\sup_{x\in\mathbb R^{[m]+1}}\;
\Bigl|\int_{\mathbb R^{[m]+1}}
2^{-j/m}\,u_{{\rm L},n}\bigl(t,|x-2^{-j}y|\bigr)\,
\psi^{\vee}(y)\,dy\Bigr|
\geq \delta_{0}\,,
\end{equation}
where 
\[\delta_{0}=\frac12c_{0}^{\frac{m+1}{(1-\theta)m}}
\Bigl(C^{\frac{\theta}{1-\theta}}C_{0}\Bigr)^{-\frac{m+1}{m}}
C_{m}^{-1/m}A^{-\frac{m+1}{m^{2}}
\bigl(\frac{\theta}{1-\theta}+\frac{1}{m+1}\bigr)}>0\,,\]
and $C_m$ is the constant in \eqref{11}.

As a result of \eqref{Lower-bound-besov}, we have a family of $(t^{0}_{n})_{n}$ in 
$\mathbb R^{\mathbb N}$, a sequence of $(j_{n})_{n}\in\mathbb Z^{\mathbb N}$ and $(x_{n})_{n}$ 
in $(\mathbb R^{[m]+1})^{\mathbb N}$such that 
\[
\Bigl|\int_{\mathbb R^{[m]+1}}
2^{-j_{n}/m}\,
u_{{\rm L},n}\bigl(t^{0}_{n},|x_{n}-2^{-j_{n}}y|\bigr)\,
\psi^{\vee}(y)\,dy\Bigr|
\geq \delta_{0}/2\,,\;n\geq N\,.\]
Denote by $\varphi(\cdot)=\psi^{\vee}(\cdot)$,
$\lambda_{n}=2^{j_{n}}$, 
$t_{n}=-t^{0}_{n}\lambda_{n}$,  and $y_{n}=\lambda_{n} x_{n}$,
we will obtain a contradiction by letting $n\rightarrow \infty$ provided, up to some subsequences,
\begin{equation}
\label{reduced}
\int_{\mathbb R^{[m]+1}}
\frac{1}{\lambda_{n}^{1/m}}u_{{\rm L},n}\Bigl(\frac{-t_{n}}{\lambda_{n}},\frac{|y-y_{n}|}{\lambda_{n}}\Bigr)\varphi(y)\,dy\longrightarrow0\,,\;n\rightarrow+\infty\,.
\end{equation}
To prove this, we
 divide the argument
into two cases.\medskip

$\bullet$ {\bf Case 1.}
$\displaystyle\limsup_{n\to\infty}|y_{n}|= +\infty$.
Up to a subsequence, we may assume 
\begin{equation}
\label{LM}
0<|y_{1}|\ll |y_{2}| \ll \cdots \ll | y_{n}| \ll |y_{n+1}| \cdots\longrightarrow+\infty\,,\,n\rightarrow +\infty.
\end{equation}
Denote by
\[V_{n}(y)=\frac{1}{\lambda_{n}^{1/m}}u_{{\rm L},n}\Bigl(\frac{-t_{n}}{\lambda_{n}},\frac{|y|}{\lambda_{n}}\Bigr)\,.\]
Note that $V_n$ is a radial function on $\RR^{[m]+1}$.
Then from the radial Sobolev embedding (see \eqref{P4} in Proposition \ref{P:prelim}), we have 
\begin{equation}
\label{radial-sobolev}
|V_{n}(y)|\leq 
\frac{1}{|y|^{1/m}}\Bigl(\int_{0}^{+\infty}|r\partial_{r}u_{{\rm L},n}(-t_{n}/\lambda_{n},r)|^{m}dr\Bigr)^{\frac1m}
\leq C_{m}\Bigl(\frac{A}{|y|}\Bigr)^{1/m}\,,
\end{equation}
for all $n$. As a consequence, \eqref{reduced} is bounded by
\begin{equation}
\label{LL}
c_{n}:=\int_{\mathbb R^{[m]+1}}
|y-y_{n}|^{-1/m}|\varphi(y)|\,dy
\,,
\end{equation}
and it suffices to show 
\begin{equation}
\label{eq:KK}
\lim_{n\rightarrow+\infty}c_{n}=0\,.
\end{equation}
We write
$$c_n= \int_{|y-y_n|\leq 1}
|y-y_{n}|^{-1/m}|\varphi(y)|\,dy+\int_{|y-y_n|\geq 1}
|y-y_{n}|^{-1/m}|\varphi(y)|\,dy.$$
The first term is bounded by 
$$ \left(\sup_{|y-y_n|\leq 1}|\varphi(y)|\right)\int_{|z|\leq 1}|z|^{-\frac 1m}\underset{n \to \infty}{\longrightarrow}0,$$
while the second one goes to zero by  dominated convergence.
Hence \eqref{eq:KK}.
\medskip

$\bullet$ {\bf Case 2.} There exists $c>0$ such that
$|y_{n}|\leq c<+\infty$ for all $n$. 
We have, up to some subsequences, $y_{n}\rightarrow y_{*}$ as $n\rightarrow \infty$, where $y_{*}\in \mathbb R^{[m]+1}$ such that $|y_{*}|\leq c$. 
Denoting by $\tau_{n}\varphi(\cdot)=\varphi(\cdot+y_{n}) $ and $\tau_{*}\varphi(\cdot)=\varphi(\cdot+y_{*})$ , we have
\begin{equation}
\label{cvg-of-phi}
\tau_{n}\varphi\longrightarrow \tau_{*}\varphi\,,\;
n\rightarrow+\infty\,,\;\text{in}\;\mathcal S(\mathbb R^{[m]+1})\,.
\end{equation}
From the the condition that
\eqref{asmp} converges weakly to zero in $\LLL^m$, 
we have
\[
\lim_{n\rightarrow+\infty}
\int_{\mathbb R^{[m]+1}}
V_{n}(x)\,\tau_{*}\varphi(x)dx= 0\,.\]
In fact, considered as a function on $\RR^3$, we have, by \eqref{P2} in Proposition \ref{P:prelim},
$$V_n\xrightharpoonup[n\to\infty]{} 0\text{ weakly in }L^{3m}(\RR^3).$$
Furthermore,
\begin{multline*}
\int_{\mathbb R^{[m]+1}}
V_{n}(x)\,\tau_{*}\varphi(x)dx=\int_0^{+\infty} \int_{S^{[m]}} \tau_*\varphi(r\omega)\,d\sigma(\omega)\,V_n(r)r^{[m]}dr\\
=\int_0^{+\infty} \underbrace{\left(\int_{S^{[m]}} \tau_*\varphi(r\omega)\,d\sigma(\omega)\,r^{[m]-2}\right)}_{:=\Psi(r)}V_n(r) r^{2}dr \underset{n\to\infty}{\longrightarrow}0,
\end{multline*}
since $\Psi(r)$ can be considered as a radial function in $L^{(3m)'}(\RR^3)$ for $1<m<+\infty$.
%

%
%
On the other hand, we have
by the fundamental theorem of calculus and integration by parts
\begin{align*}
&\int_{\mathbb R^{[m]+1}}V_{n}(|y|)\,\bigl(\tau_{n}\varphi(y)-\tau_{*}\varphi(y)\bigr)dy\\
=&\int_{0}^{1}\int_{\mathbb R^{[m]+1}}\left\langle\nabla V_{n}(y)\,,
(y_{*}-y_{n})\varphi\bigl(y+s(y_{n}-y_{*})+y_{*}\bigr)\right\rangle dyds\,.
\end{align*}
After using H\"older's inequality and the energy estimate, we see the term on the righthand side is bounded by
\[
C_{m}A^{\frac1m}|y_{n}-y_{*}|\int_{0}^{1}
\Bigl(\int_{\mathbb R^{[m]+1}}
\bigl|\varphi(y+s(y_{n}-y_{*})+y_{*})\bigr|^{\frac{m}{m-1}}
|y|^{-\frac{m-[m]}{m-1}}dy\Bigr)^{\frac{m-1}{m}}ds\,.
\]
Notice that $\varphi\in \mathcal S(\mathbb R^{[m]+1})$, $|y_{*}|\leq c$
and $|y|^{-(m-[m])/(m-1)}$ is integrable near the origin of $\mathbb R^{[m]+1}$
when $m>1$. We have
\[
\lim_{n\rightarrow\infty}\int_{\mathbb R^{[m]+1}}V_{n}(y)\,\bigl(\tau_{n}\varphi(y)-\tau_{*}\varphi(y)\bigr)dy=0\,.
\]
\end{proof}

\subsection{Bessel-type inequality}
\label{SS:Bessel}
In this subsection we prove Proposition \ref{P:Pythagorean}.

We let $\{u_{{\rm L},n}\}_{n\in\mathbb N}$, $J\geq 1$, and, for $j\in\{1,\ldots,J\}$, $U^j_{\rm L}$, $(\lambda_{j,n},t_{j,n})_n$ be as in Proposition \ref{P:Pythagorean}, and define $U^j_{{\rm L},n}$ by \eqref{ULnJ} and $w_{n}^J$ by \eqref{wnJ}.

First of all, 
we have the explicit formula for 
$\bigl[ U^{j}_{\rm L} \bigr]_{\pm}(t,r)$:
\begin{equation}
\label{eq:U+-}
\bigl[ U^{j}_{\rm L} \bigr]_{+}(t,r)
=2\dot{F}^{j}(t+r)\,,\;
\bigl[ U^{j}_{\rm L} \bigr]_{-}(t,r)
=2\dot{F}^{j}(t-r)\,,
\quad j\geq 1\,,
\end{equation}
with 
\[
F^{j}(\sigma)=\frac12\sigma U_{0}^{j}(|\sigma|)
+\frac12\int_{0}^{|\sigma|}\varrho \, U_{1}^{j}(\varrho)\,d\varrho.
\]
In view of \eqref{def_pm},
 one easily verifies that
\[
\bigl[U^{j}_{{\rm L},\,n}\bigr]_{\pm}(t,r)
=\frac{1}{\lambda_{j,\,n}^{\frac1m}}
\bigl[U^{j}_{\rm L}\bigr]_{\pm}
\Bigl(\frac{t-t_{j,n}}{\lambda_{j,n}},
\frac{r}{\lambda_{j,n}}\Bigr)\,.
\]
Up to subsequences, we may assume, after translating in time
and rescaling $U^{j}_{{\rm L}}$ if necessary
\begin{equation}
\label{eq:reduced-paramtrs}
j\geq 1,\,\lim_{n\rightarrow\infty}-\frac{t_{j,n}}{\lambda_{j,n}}=\pm\infty\quad
\text{or}\quad
\forall \,n,\; t_{j,n}=0\,.
\end{equation}
\subsubsection*{Step 1. Decoupling of linear profiles} In this step, we prove
\begin{equation}
\label{eq:decoupling of profiles}
\lim_{n\rightarrow+\infty}
E_m\left(\sum_{j=1}^{J}\overrightarrow{U}^j_{{\rm L},n}(0)\right)=\sum_{j=1}^{J}E_m\left(\overrightarrow{U}^j_{\rm L}(0)\right)
\end{equation}
Recall that for any solution $u$ of the linear wave equation, we have 
$$E_m(\vec{u}(0))=E_m(\vec{u}(t))= \sum_{\pm}\int_0^{+\infty} \left| [u]_{\pm}(t,r)\right|^m\,dr,$$
where $[u]_{\pm}$ is defined in (\ref{def_pm}). Hence (for constants $C>0$ that depend on $J$ and $m$, but not on $n$)
\begin{multline*}
\left| E_m\left( \sum_{j=1}^J U^j_{{\rm L},n}(0)\right) -\sum_{j=1}^J E_m\left(U^j_{\rm L}(0)\right)\right|
\\
=\left| E_m\left( \sum_{j=1}^J U^j_{{\rm L},n}(0)\right) -\sum_{j=1}^J E_m\left(U^j_{{\rm L},n}(0)\right)\right|\\
\leq C\sum_{\substack{ j\neq k\\ \pm}} \int_{0}^{+\infty} \left| \left[ U^j_{{\rm L},n}\right]_{\pm}(0,r)\right|^{m-1}\left| \left[ U^k_{{\rm L},n}\right]_{\pm}(0,r)\right|\,dr\\
\leq C\sum_{\substack{ j\neq k\\ \pm}} \underbrace{\int_{0}^{+\infty} \left|\frac{1}{\lambda_{j,n}^{\frac{1}{m}}} \dot{F}^j\left(\frac{-t_{j,n}\pm r}{\lambda_{j,n}}\right) \right|^{m-1}\left|\frac{1}{\lambda_{k,n}^{\frac{1}{m}}} \dot{F}^k\left(\frac{-t_{k,n}\pm r}{\lambda_{k,n}}\right)  \right|\,dr}_{I^{\pm}_{j,k,n}}
\end{multline*}
We are thus reduced to proving that each of the term $I_{j,k,n}^{\pm}$ ($j\neq k$) goes to $0$ as $n$ goes to infinity. 
By density we may assume
$$
U^{j}_{0},\, U^{j}_{1},\, U^{k}_{0},\, U^{k}_{1}\in C^{\infty}_{0},
$$
and thus $\dot{F}^j,\,\dot{F}^k\in C^{\infty}_0$. We will only consider $I_{j,k,n}^+$, whereas the proof for $I_{j,k,n}^-$ is the same. Extracting subsequences and arguing by contradiction, we can distinguish without loss of generality between the following three cases.

\EMPH{Case 1} We assume
$\lim_{n\to\infty} \frac{\lambda_{k,n}}{\lambda_{j,n}}=0.$
By the change of variable $s=\frac{-t_{k,n}+r}{\lambda_{k,n}}$, we obtain
\begin{multline}
\label{Ijkn}
I_{j,k,n}^{+}=\int_{-t_{k,n}/\lambda_{k,n}}^{+\infty} \left(\frac{\lambda_{k,n}}{\lambda_{j,n}}\right)^{1-\frac 1m} \left| \dot{F}^j\left(\frac{\lambda_{k,n}s+t_{k,n}-t_{j,n}}{\lambda_{j,n}}\right)\right|^{m-1} \left|\dot{F}^k(s)\right|\,ds\\
\lesssim \left(\frac{\lambda_{k,n}}{\lambda_{j,n}}\right)^{1-\frac{1}{m}},
\end{multline}
where we have used that $\dot{F}^j$ and $\dot{F}^k$ are bounded and compactly supported. Since $\frac{\lambda_{k,n}}{\lambda_{j,n}}$ goes to $0$ as $n$ goes to infinity, we are done.

\EMPH{Case 2} We assume
$\lim_{n\to\infty} \frac{\lambda_{j,n}}{\lambda_{k,n}}=0.$
 We argue similarly by using the change of variable $s=\frac{-t_{j,n}+r}{\lambda_{j,n}}$.

\EMPH{Case 3} We assume that the sequence $\left(\lambda_{j,n}/\lambda_{k,n}+\lambda_{k,n}/\lambda_{j,n}\right)_n$ is bounded. By the pseudo-orthogonality condition (\ref{P-O}) and formula (\ref{Ijkn}), we see that $I_{j,k,n}^+$ is $0$ for large $n$, which concludes Step 1.

\subsubsection*{Step 2. End of the proof}
For $1<m<+\infty$, we introduce the notation
\begin{align*}
\Phi^{j}_{n,0}(r)=&\frac{1}{2r}\sum_{\pm}\int_{0}^{r}
\bigl|\bigl[U^{j}_{{\rm L},\,n}\bigr]_{\pm}(0,s)\bigr|^{m-2}
\bigl[U^{j}_{{\rm L},\,n}\bigr]_{\pm}(0,s)\,ds
\\
\Phi^{j}_{n,1}(r)=&\frac{1}{2r}\sum_{\pm}
\pm \bigl|\bigl[U^{j}_{{\rm L},\,n}\bigr]_{\pm}(0,r)\bigr|^{m-2}
\bigl[U^{j}_{{\rm L},\,n}\bigr]_{\pm}(0,r)
\end{align*}
and let $\Phi^{j}_{n,\rm L}(t)$ be the  solution of the linear wave equations with initial data $(\Phi^{j}_{n,0},\Phi^{j}_{n,1})\in \mathcal{L}^{m'}$ where $m'=\frac{m}{m-1}$.
Then we have
\begin{align*}
\bigl[\Phi^{j}_{n,\rm L}\bigr]_{\pm}(0,r)=&\bigl|\bigl[ U^j_{{\rm L},n}\bigr]_\pm(0,r)\bigr|^{m-2}\bigl[U^j_{{\rm L},n}\bigr]_\pm(0,r),\\
\end{align*}
and note that:
\begin{equation}
\label{eq:duality-n}
\begin{split}
E_m(\overrightarrow{U}^j_{\rm L}(0))
=E_m(\overrightarrow{U}^j_{{\rm L},n}(0))
=\int_{0}^{+\infty}\sum_{\pm}
\bigl[\Phi^{j}_{n,\rm L}\bigr]_\pm(0)\bigl[U^j_{{\rm L},n}\bigr]_\pm(0)\,
dr.
\end{split}
\end{equation}
From the weak convergence condition satisfied by the remainder term $w^J_n$, we have
by time translation and changing  variables
\begin{align*}
&\int_{0}^{+\infty}
\Bigl(
\bigl[\Phi^{j}_{n,\rm L}\bigr]_{+}(0,r)\,
\bigl[w^{J}_{n}\bigr]_{+}(0,r)
+
\bigl[\Phi^{j}_{n,\rm L}\bigr]_{-}(0,r)\,
\bigl[w^{J}_{n}\bigr]_{-}(0,r)\Bigr)\,dr
\\
=&\int_{0}^{+\infty}
\bigl|\bigl[U^{j}_{\rm L}\bigr]_{+}(0,r)\bigr|^{m-2}
\bigl[U^{j}_{\rm L}\bigr]_{+}(0,r)\,
\lambda_{j,n}^{\frac1m}\bigl[w^{J}_{n}\bigr]_{+}
(t_{j,n},\,\lambda_{j,n}\,r)\,dr\\
&\qquad+
\int_{0}^{+\infty}
\bigl|\bigl[U^{j}_{\rm L}\bigr]_{-}(0,r)\bigr|^{m-2}
\bigl[U^{j}_{\rm L}\bigr]_{-}(0,r)\,
\lambda_{j,n}^{\frac1m}\bigl[w^{J}_{n}\bigr]_{-}
(t_{j,n},\,\lambda_{j,n}\,r)\,dr\,,
\end{align*}
which goes to zero as $n\rightarrow+\infty$ for $1\leq j\leq J$. Furthermore,
$$ \int_0^{+\infty} \left| \left[\Phi_{n,\rm L}^{j}\right]_{\pm}(0,r)  \left[ U_{{\rm L},n}^k\right]_{\pm}(0,r)\right|dr= \int_0^{+\infty} \left| \left[U_{{\rm L},n}^{j}\right]_{\pm}(0,r) \right|^{m-1} \left|\left[ U_{{\rm L},n}^k\right]_{\pm}(0,r)\right|dr,$$
and, by Step 1, this goes to $0$ as $n$ goes to infinity if $j\neq k$.
Hence  from \eqref{eq:duality-n}, we have 
\begin{align*}
\sum_{j=1}^{J}E_m(\overrightarrow{U}^j_{\rm L}(0))
=&\lim_{n\rightarrow+\infty}\Biggl[
\int_{0}^{+\infty}
[u_{{\rm L},n}]_+(0,r)\Bigl(\sum_{j=1}^{J}\bigl[\Phi^{j}_{n,\rm  L}\bigr]_{+}(0,r)\Bigr)
dr\\
&\qquad\qquad+\int_{0}^{+\infty}
[u_{{\rm L},n}]_-(0,r)\Bigl(\sum_{j=1}^{J}\bigl[\Phi^{j}_{n,\rm  L}\bigr]_{-}(0,r)\Bigr)
dr\Biggr],
\end{align*}
which is bounded after using H\"older's inequality by
\[
\left[\lim_{n\rightarrow+\infty} E_{m'}\left(\sum_{j=1}^{J}\overrightarrow{\Phi^{j}_{n,\rm  L}}(0,r)\right) \right]^{1/m'}
\left[\limsup_{n\rightarrow+\infty}
E_m(\overrightarrow{u}_{{\rm L},n}(0))\right]^{1/m}.
\]
Furthermore, by the decoupling property proved in Step 1 we obtain
\[\lim_{n\rightarrow+\infty} E_{m'}\left(\sum_{j=1}^{J}\overrightarrow{\Phi^{j}_{n,\rm  L}}(0,r)\right)=\sum_{j=1}^{J}E_{m'}(\overrightarrow{\Phi_{n,\rm L}^{j}}(0)) =\sum_{j=1}^{J}E_{m}(\overrightarrow{U}^j_{\rm L}(0))\]
and this concludes the result.

\subsection{Approximation by sum of profiles}
\label{SS:approx}
We next write a lemma approximating a nonlinear solution by a sum of profiles outside a wave cone. This type of approximation is only available in space-time slabs where the $S$ norm of all the profiles remain finite. To satisfy this assumption, we will work outside a sufficiently large wave cone. 

Let $\big((u_{0,n},u_{1,n})\big)_n$ be a sequence of functions in $\LLL^m$ that has a profile decomposition with profiles $(U^j_0, U^j_1)$ and parameters $(\lambda_{j,n},t_{j,n})_n$, $j \geq 1$. Extracting subsequences and time-translating the profiles, we can assume that for all $j\geq 1$ one of the following holds:
\begin{gather}
 \label{profile_infty}
 \lim_{n\to\infty} -t_{j,n}/\lambda_{j,n}\in \{\pm \infty\}\text{ or}\\
 \label{profile_0}
 \forall n,\quad t_{j,n}=0.
\end{gather}
We will denote by $\JJJ_{\infty}$ the set of indices $j$ such that \eqref{profile_infty} holds and $\JJJ_0$ the set of indices such that \eqref{profile_0} holds. We assume
\begin{enumerate}
 \item \label{I:j0} There exists $j_0\geq 1$, $A>0$ and a global solution $U^{j_0}$ of 
 $$\left\{
 \begin{aligned}
 \partial_t^2U^{j_0}-\Delta U^{j_0}&=\iota |U^{j_0}|^{2m}U^{j_0}\indic_{\{r\geq |t|+A\}}\\
  \vec{U}^{j_0}(0,r)&=\vec{U}^{j_0}_{\rm L}(0,r),\quad r\geq A.
 \end{aligned}
\right. 
 $$
 such that $\vec{U}^{j_0}(0)\in \LLL^m$ and $\|U^{j_0}\|_{S(\{r\geq |t|+A\})}<\infty$.
 \item \label{I:JJJ0}If $j\in \JJJ_0\setminus \{j_0\}$, then the solution of \eqref{NLW} with initial $\vec{U}^j_{\rm L}(0)$ scatters in both time direction or
 $$\lim_{n\to\infty} \frac{\lambda_{j,n}}{\lambda_{j_{0,n}}}=0.$$
\end{enumerate}
For $j\geq 1$, we define $U^j$ as follows:
\begin{itemize}
 \item $U^{j_0}$ is defined as in point \eqref{I:j0} above.
 \item if $j\in \JJJ_0$ and $\lim_{n\to\infty} \frac{\lambda_{j,n}}{\lambda_{j_{0,n}}}=0$, then $U^j$ is the solution of \eqref{NLW} with initial data $\vec{U}^j_{\rm L}(0)$.
 \item if $j\in \JJJ_0$ and $\lim_{n\to\infty} \frac{\lambda_{j,n}}{\lambda_{j_{0,n}}}=\infty$, then $U^j=0$. 
 \item if $j\in \JJJ_{\infty}$, then $U^j=U^j_{\rm L}$.
\end{itemize}
We let $U^j_n$ be the corresponding modulated profiles:
$$U^j_n(t,x)=\frac{1}{\lambda_{j,n}^{1/m}}U^j\left( \frac{t-t_{j,n}}{\lambda_{j,n}},\frac{x}{\lambda_{j,n}} \right).$$
\begin{lemma}
\label{L:approx}
Assume that points \eqref{I:j0} and \eqref{I:JJJ0} above hold, let $u_n$ be the solution of \eqref{NLW} with initial data $(u_{0,n},u_{1,n})$, and $I_n$ be its maximal interval of existence. Then 
$$u_n(t,x)=\sum_{j=1}^J U^j_{n}(t,x)+w_n^J(t,x)+\eps_n^J(t,x),$$
where 
$$\lim_{J\to\infty}\limsup_{n\to\infty}\left(\left\|\eps^J_n\right\|_{S(\{t\in I_n,\; r\geq A\lambda_{j_0,n}+|t|\})}+\sup_{t\in I_n}\int_{A\lambda_{j_0,n}+|t|}^{+\infty}\left|r\partial_{t,r}\vec{\eps}_n^J(t,r)\right|^m\,dr\right)=0.$$
\end{lemma}
\begin{proof}
 This follows from Lemma \ref{L:ELTPT} with 
 $$\tilde{u}_n=\sum_{j\in \JJJ_0}U^j_n.$$
 We omit the details of the proof that are by now standard (see e.g. the proof of the main theorem in \cite{BaGe99}).
\end{proof}

\subsection{Exterior energy of a sum of profiles}
\begin{proposition}
\label{P:exterior_profiles}
Let $\{ (u_{0,n},u_{1,n}) \}_{n \in \mathbb{N}}$ be a bounded sequence in $\LLL^m$
that has a profile decomposition with profiles $\{ U_{\rm L}^{j} \}_{j\geq 1}$ and parameters
$ \left\{ (t_{j,n}, \lambda_{j,n})_n \right\}_{j\geq 1}$. Let
$ \left\{ (\theta_{n},\rho_{n},\sigma_n)\right\}_{n \in \mathbb{N}}$ be a sequence such that $ 0 \leq \rho_{n}<\sigma_n\leq\infty$, $\theta_n\in \RR$. Let $k\geq 1$. Then, extracting a subsequence if necessary
\begin{equation}
\label{Eqn:Orth}
o_n(1)+\int_{\rho_n}^{\sigma_n}|r\partial_{r,t} u_{{\rm L},n}(\theta_n,r)|^m\,dr
\\
\geq \int_{\rho_n}^{\sigma_n}\left|r\partial_{r,t} U_{{\rm L},n}^k(\theta_n,r)\right|^m\,dr
\end{equation}
where $\lim_n o_n(1)=0$, $u_{{\rm L},n}$ is the solution of the linear wave equation with initial data $(u_{0,n},u_{1,n})$ and $U_{{\rm L},n}^k$ is defined in \eqref{ULnJ}.
%
\end{proposition}
(see \cite[Proposition 3.12]{DuRoy15P} for the proof.) 

\section{Exterior energy for solutions of the nonlinear equation}
\label{S:exterior}
\subsection{Preliminaries on singular stationary solutions}
We recall from \cite{DuKeMe12c,DuRoy15P,Shen13} the following result on existence of stationary solutions for equation \eqref{NLW}
\begin{proposition}
\label{P:stat}
Let $\ell \in \mathbb{R}\setminus \{0\}$. Assume $m>1$, $m\neq 2$. There exists $R_{\ell}\geq 0$ and a maximal radial $C^{2}$ solution $Z_{\ell}$ of 
\begin{equation}
\label{Zell1}
\Delta Z_{\ell}  + \iota |Z_{\ell}|^{2m} Z_{\ell} = 0 \quad\text{on}\quad\mathbb{R}^{3} \cap \{ |x|>R_{\ell}\}
\end{equation}
such that
\begin{equation}
\label{Zell2}
|r Z_{\ell}(r)  - \ell|+\left| r^2Z_{\ell}'(r)+\ell\right| \lesssim \frac{1}{r^{2m-2}}, \quad r \gg 1.
\end{equation}
Furthermore
\begin{itemize}
 \item if $\iota=+1$ (focusing nonlinearity),  $R_{\ell}=0$ and $Z_{\ell} \notin L^{3m}(\R^3)$.
\item if $\iota=-1$ (defocusing nonlinearity), $R_{\ell}>0$ and 
\begin{equation}
\label{Eqn:limZl}
\lim_{r\to R_{\ell}}|Z_{\ell}(r)|=+\infty.
\end{equation}
\end{itemize}
\end{proposition}
\begin{remark}
\label{R:Zscaling}
We will construct $Z_1$ and let 
$$ Z_{\ell}=\frac{\pm 1}{|\ell|^{\frac{1}{m-1}}}Z_1\left( \frac{r}{|\ell|^{\frac{m}{m-1}}}\right),$$
(where $\pm$ is the sign of $\ell$) that satisfies the conclusion of Proposition \ref{P:stat} for all $\ell\in \RR\setminus\{0\}$. In particular:
$$ R_{\ell}=R_1|\ell|^{\frac{m}{m-1}}.$$
Let us mention that the uniqueness of $Z_{\ell}$ can be proved by elementary arguments. However it will follow from Proposition \ref{P:BB1} and we will not prove it here.
\end{remark}

\begin{proof}
The proof is essentially contained in \cite{DuKeMe12c,Shen13} (focusing case for $m>2$ and $m\in (1,2)$ respectively) and \cite{DuRoy15P} (defocusing case for $m>2$). We give a sketch for the sake of completeness.

We assume $\ell=1$ (see Remark \ref{R:Zscaling}).

\emph{Existence for large $r$}.
Letting $g=rZ_1$, we see that the equation on $Z_1$ is equivalent to 
\begin{equation}
 \label{eq_g}
 g''(r)=-\frac{\iota}{r^{2m}}|g(r)|^{2m}g(r),
\end{equation} 
it is sufficient to find a fixed point for the operator $A$ defined by
$$A(g)=1-\int_r^{\infty} \int_{s}^{\infty}\frac{\iota}{\sigma^{2m}} |g(\sigma)|^{2m}g(\sigma)\,d\sigma\,ds$$
in the ball 
$$B=\left\{g\in C^{0}\left( [r_0,+\infty),\R \right),\quad d(g,1)\leq M\right\},$$
where $r_0$ and $M$ are two large parameters and 
$$d(g,h):=\sup_{r\geq r_0} \left( r^{2m-2}|g(r)-h(r)| \right).$$
Noting that $(B,d)$ is a complete metric space, it is easy to prove that $A$ is a contraction on $B$ assuming $M\gg 1$, and $r_0\gg 1$ (depending on $M$), and thus that $A$ has a fixed point $g_1$.  The fact that $Z_1:=\frac{1}{r}g_1$ satisfy the estimates \eqref{Zell2} follows easily. Let $R_1\geq 0$ such that $(R_1,+\infty)$ is the maximal interval of existence of $g_1$ as a solution of the ordinary differential equation. 

\emph{Focusing case}. We next assume $\iota=1$ and prove that $R_{1}=0$ and $Z_{\ell}\notin L^{3m}$. Let
$$G(r)=\frac{1}{2}g'(r)^2+\frac{1}{(2m+2)r^{2m}}|g(r)|^{2m+2}.$$
By \eqref{eq_g}, if $r\in (R_1,+\infty)$,
$$ G'(r)=-\frac{m}{(m+1)r^{2m+1}}|g(r)|^{2m+2}.$$
Hence
$$|G'(r)|\leq \frac{C}{r}\,G(r)\,.$$
This proves that $G$ is bounded on $(R_1,+\infty)$ if $R_1>0$, a contradiction with the standard ODE blow-up criterion. Thus $R_1=0$. 

The fact that $Z_1\notin L^{3m}(\R^3)$ is non-trivial but classical. Assume by contradiction that $Z_1\in L^{3m}$. Then one can prove (see \cite{DuKeMe12c}) that $Z_1$ is a solution in the distributional sense on $\R^3$ of 
$$-\Delta Z_1=|Z_1|^{2m}Z_1.$$
Noting that $|Z_1|^{2m}\in L^{3/2}$, one can use the work of Trudinger \cite{Trudinger68} to prove that $Z_1\in L^{\infty}$, and thus, by elliptic regularity, that $Z_1$ is $C^2$ on $\R^3$. To deduce a contradiction, we introduce, as in \cite{Shen13}, the function $v(r)=r^{\frac{1}{m}}Z_1$. It is easy to check, using \eqref{Zell2}, for the limits at infinity and the fact that $Z_1$ is $C^2$ for the limit at $0$, that
$$\lim_{r\to 0^+} v(r)=\lim_{r\to 0^+} rv'(r)=\lim_{r\to +\infty} v(r)=\lim_{r\to +\infty} rv'(r)=0.$$
Furthermore,
$$v''+\frac{2}{r}\left( 1-\frac{1}{m} \right)v'+\frac{1}{r^2}\left( \frac{1}{m^2}-\frac{1}{m} \right)v+\frac{1}{r^2}|v|^{2m}v=0.$$
Integrating the identity
\begin{equation}
\label{identity}
\frac{d}{dr}\left(r^2\frac{|v'(r)|^2}{2}-\frac{m-1}{2m^2}v^2(r)+\frac{|v(r)|^{2m+2}}{2m+2}\right)=\frac{2-m}{m}r|v'(r)|^2 
\end{equation} 
between $0$ and $+\infty$, one see that $v$ must be a constant, a contradiction with the construction of $Z_1$. Note that we have used in this last step that the constant $\frac{2-m}{m}$ in the right-hand side of the identity \eqref{identity} is non-zero, i.e that $m\neq 2$.

\emph{Defocusing case}. Assume $\iota=-1$. We prove that $R_1>0$ by contradiction. Assume $R_1=0$ and let 
$$h(s):=Z_{\ell}\left( \frac{1}{s} \right).$$
Then 
$$h''(s)=\frac{1}{s^4}|h(s)|^{2m}h(s)$$
and by \eqref{Zell2},
$$\lim_{s\to 0^+}\frac{h(s)}{s}=\lim_{s\to 0^+} h'(s)=1.$$
By a classical ODE argument (see \cite{DuRoy15P} for the details), one can prove that $h$ blows up in finite time, a contradiction. This proves that $R_1>0$. The condition \eqref{Eqn:limZl} follows from the standard ODE blow-up criterion.
\end{proof}

\subsection{Statement}
One of the main ingredient of the proof of Theorem \ref{T:scattering} is a bound from below of the exterior $\LLL^m$-energy for nonzero, $\LLL^m$ solutions of \eqref{NLW}. 
It is similar to \cite[Propostions 2.1 and 2.2]{DuKeMe13} and \cite[Propositions 4.1 and 4.2]{DuRoy15P}. The statements in these articles are divided between two case, whether the support of $(u_0,u_1)-(Z_{\ell},0)$ is compact for all $\ell\neq 0$ or not. We give below an unified statement.

If $(u_0,u_1)\in \LLL^m$ and $A>0$ we will denote by $\TTT_A(u_0,u_1)$ the element of $\LLL^m$ defined by
\begin{align}
\label{def_TA}
 \TTT_A(u_0,u_1)(r)&= (u_0,u_1)(r) \text{ if } r>A\\
 \TTT_A(u_0,u_1)(r)&= (u_0(A),0) \text{ if } r\leq A.
\end{align}
We note that 
\begin{equation}
 \label{norm_TA}
\left\|\TTT_A(u_0,u_1)\right\|^m_{\LLL^m}=\int_{A}^{+\infty}\left(\left|\partial_r u_0(r)\right|^m+\left|u_1(r)\right|^m\right)\,r^mdr.
 \end{equation} 
 We denote by $\esssupp$ the essential support of a function defined on a domain $D$ of $\R^3$:
 $$\esssupp(f)=D\setminus \bigcup\left\{ \Omega\subset D,\; \Omega \text{ is open and }f=0\text{ a.e. in }\Omega\right\}.$$
 Recall from Proposition \ref{P:stat} the definition of $Z_1$ and $R_1$.
\begin{proposition}
\label{P:BB1}
Let $u$ be a radial solution  of \eqref{NLW} with $(u_0,u_1)\in \LLL^m$. Assume that $(u_0,u_1)$ is not identically $0$. Then there exist $A>0$, $\eta>0$ such that, if $(\tilde{u}_0,\tilde{u}_1)=\TTT_A(u_0,u_1)$, and $\tilde{u}$ is the solution of 
\begin{equation}
\label{truncated_eq}
\partial_t^2\tilde{u}-\Delta \tilde{u}=\iota |\tu|^{2m}\tu \indic_{\{r\geq A+|t|\}}
\end{equation}
with initial data $(\tilde{u}_0,\tilde{u}_1)$, then $\tilde{u}$ is global, scatters in $\LLL^m$
and the following holds for all $t\geq 0$ or for all $t\leq 0$:
\begin{equation}
 \label{BB2}
 \int_{A+|t|}^{+\infty}|\partial_r \tu(r)|^mr^m\,dr+\int_{A+|t|}^{+\infty}|\partial_t \tu(r)|^mr^m\,dr\geq \eta.
\end{equation} 
\end{proposition}
The proof of Proposition \ref{P:BB1} is very close to the proofs of the analogous propositions in \cite{DuKeMe12c} and \cite{DuRoy15P}. We give a sketch of proof for the sake of completeness.
\subsection{Sketch of proof of Proposition \ref{P:BB1}}
We argue by contradiction, assuming that for all $A>0$ the solution $\tilde{u}$ of \eqref{truncated_eq} with initial data $\TTT_A(u_0,u_1)$ is not a scattering solution, or is scattering and satisfies
\begin{equation}
 \label{BB2'}
 \liminf_{t\to\pm\infty} \int_{A+|t|} |\partial_{t,r} \tilde{u}(t,r)|^mr^m\,dr=0.
\end{equation} 
We let 
$$v(r)=ru(r),\quad v_0(r)=ru_0(r),\quad v_1(r)=ru_1(r).$$
\EMPH{Step 1} In this step we prove that there exists $\eps_0>0$ such that, if $A>0$ is such that 
\begin{equation}
 \label{BB3}
 \int_{A}^{+\infty}\left(|\partial_ru_0|^m+|u_1|^m\right)r^m\,dr =\eps\leq \eps_0,
\end{equation} 
then
\begin{gather}
\label{BB4}
\int_A^{+\infty} |\partial_r v_0|^m+|v_1|^m\,dr\leq \frac{C}{A^{(2m+1)(m-1)}}|v_0(A)|^{m(2m+1)},\\
\label{BB5}
\forall B\in [A,2A], \quad \left|v_0(B)-v_0(A)\right|\leq CA^{2-2m} |v_0(A)|^{2m+1}\leq C\eps^2 |v_0(A)|.
\end{gather}
We first assume \eqref{BB4} and prove \eqref{BB5}. By H\"older inequality and \eqref{BB4} we have
\begin{multline}
\label{BB4'}
 \left|v_0(B)-v_0(A)\right|\leq\int_{A}^{2A} |\partial_r v_0(r)|\,dr \leq A^{\frac{m-1}{m}}\left(\int_A^{2A}|\partial_r v_0|^m\,dr\right)^{\frac{1}{m}}\\
 \leq CA^{2-2m} |v_0(A)|^{2m+1}.
\end{multline}
Furthermore, by \eqref{BB3} and \eqref{P4} in Proposition \ref{P:prelim},
$$\frac{1}{A^{m-1}}|v_0(A)|^m=A|u_0(A)|^m\lesssim \eps,$$
which yields
$$ |v_0(A)|^{2m}\lesssim \eps^2A^{2m-2}.$$
Combining with \eqref{BB4'}, we obtain the second inequality of \eqref{BB5}.

We next prove \eqref{BB4}. Let
$$(\tu_0,\tu_1)=\TTT_A(u_0,u_1).$$

Let $\tu$ (respectively $\tu_{\rm L}$) be the solution of the nonlinear wave equation \eqref{NLW} (respectively the linear wave equation \eqref{LW}) with initial data $(\tu_0,\tu_1)$. By the small data theory, $\tu$ is global and 
\begin{equation}
 \label{BB6}
 \sup_{t\in \R} \left\|\vec{\tu}(t)-\vec{\tu}_{\rm L}(t)\right\|_{\LLL^m}\leq C\eps^{2m+1}.
\end{equation} 
Using the exterior energy property \eqref{P6} in Proposition \ref{P:prelimLW}, we have that the following holds for all $t\geq 0$ or for all $t\leq 0$:
\begin{multline*}
\int_A^{+\infty} \left(|\partial_r(v_0)|^m+|v_1|^m  \right)\,dr\leq C\int_{A+|t|}^{+\infty} \left|\partial_{r,t}(r\tilde{u}_{\rm L})(t,r) \right|^m\,dr \\
\leq C\int_{A+|t|}^{+\infty} \left|\partial_{r,t}\tilde{u}_{\rm L}(t,r) \right|^mr^m\,dr
\end{multline*}
Using \eqref{BB6}, we obtain that the following holds for all $t\geq 0$ or for all $t\leq 0$:
\begin{equation}
 \label{BB7}
 \int_A^{+\infty} \left(|\partial_r(v_0)|^m+|v_1|^m  \right)\,dr\leq C\left(\int_{A+|t|}^{+\infty} |\partial_{r,t}\tilde{u}(t,r)|^m\,dr+\eps^{(2m+1)m}\right).
\end{equation}
Using \eqref{BB2'} and the definition \eqref{BB3} of $\eps$, and letting $t\to+\infty$ or $t\to-\infty$, we obtain
$$\frac{1}{C}\int_A^{+\infty} \left(|\partial_rv_0|^m+|v_1|^m\right)dr\leq \left( \int_A^{+\infty}\left( |\partial_ru_0|^m+|u_1|^m \right)r^m\,dr  \right)^{2m+1}.$$
By \eqref{P4} in Proposition \ref{P:prelim}, and since $A |u_0(A)|^m=\frac{1}{A^{m-1}} |v_0(A)|^m$,
$$\int_A^{+\infty}\left( |\partial_r v_0|^m+|v_1|^m \right)\,dr\leq C\left( \int_{A}^{+\infty} (|\partial_r v_0|^m +|v_1|^m)\,dr+\frac{1}{A^{m-1}} |v_0(A)|^{m} \right)^{2m+1}.$$
Since 
$\int_A^{+\infty}\left( |\partial_rv_0|^m+|v_1|^m \right)\,dr$
is small, we deduce \eqref{BB4}.

\EMPH{Step 2} We prove that there exists $\ell \in \R\setminus 0$ such that 
\begin{equation}
 \label{BB8}
 \lim_{r\to\infty} v_0(r)=\ell,
\end{equation} 
and that there exists a constant $M>0$ (depending on $u$) such that 
\begin{equation}
 \label{BB9}
 |v_0(r)-\ell|\leq \frac{M}{r^{2m-2}}
\end{equation} 
for large $r$.

Let $\eps>0$ and fix $A_0$ such that 
\begin{equation}
 \label{cond_A0}
 \int_{A_0}^{+\infty} \left( |\partial_r u_0|^m+|u_1|^m \right)r^m\, dr= \eps\leq \eps_0,
\end{equation} 
where $\eps_0$ is given by Step 1. By \eqref{BB5}, 
$$\forall k\geq 0,\quad \left|v_0(2^{k+1}A_0)\right|\leq \left( 1+C\eps^2 \right)\left( |v_0(2^kA_0)| \right).$$
Hence, by a straightforward induction,
$$\forall k\geq 0,\quad \left|v_0\left( 2^{k+1}A_0 \right)\right|\leq \left( 1+C\eps^2 \right)^k\left|v_0\left( A_0 \right)\right|.$$
Using \eqref{BB5} again, we deduce
\begin{equation}
 \label{BB10}
 \left|v_0\left(2^{k+1}A_0\right)-v_0\left(2^kA_0  \right)\right|\leq C\left( 2^k A_0\right)^{2-2m} (1+C\eps^2)^{k(2m+1)} |v_0(A_0)|^{2m+1}.
\end{equation} 
Choosing $\eps$ small enough (so that $2^{2-2m}(1+C\eps^2)^{2m+1}<1$), we see that 
$$\sum_{k\geq 1}\left|v_0\left(2^{k+1}A_0\right)-v_0\left(2^kA_0  \right)\right|<\infty,$$ and thus that $v_0(2^kA_0)$ has a limit $\ell$ as $k\to +\infty$. Using \eqref{BB5} again, we deduce 
\begin{equation*}
 \lim_{r\to\infty} |v_0(r)|=\ell.
\end{equation*} 
Summing \eqref{BB10} over all $k\geq 0$, we deduce, using that $v_0$ is bounded, that there exists a constant $M>0$, such that  $|v_0(A_0)-\ell|\leq MA_0^{2-2m}$ for $A_0$ large enough. This yields \eqref{BB9}. 

It remains to prove that $\ell\neq 0$. We argue by contradiction. By \eqref{BB9}, if $\ell=0$, then 
$$|v_0(r)|\leq \frac{M}{r^{2m-2}}.$$
On the other hand, using \eqref{BB5} and an easy induction argument, we obtain that for all $\eps>0$, for all $A_0$ satisfying \eqref{cond_A0},
$$|v_0(2^kA_0)|\geq (1-C\eps^2)^k|v_0(A_0)|.$$
Combining with the previous bound, we obtain 
$$(1-C\eps^2)^k|v_0(A_0)|\leq \frac{M}{(2^kA_0)^{2m-2}},$$
a contradiction if $\eps$ is chosen small enough unless $v_0(A_0)=0$. Using \eqref{BB4}, we see that this would imply $v_0(r)=0$ and $v_1(r)=0$ for almost all $r\geq A_0$. Since this is true for any $A_0$ such that \eqref{cond_A0} holds, an obvious bootstrap argument proves that $(v_0,v_1)=(0,0)$ almost everywhere, contradicting our assumption.

\EMPH{Step 3} Recall from Proposition \ref{P:stat} the definition of $R_{\ell}$. Let, for $r>R_{\ell}$, 
$$(g_0,g_1)(r):=(u_0(r)-Z_{\ell}(r),u_1(r))\quad (h_0,h_1)(r)=r(g_0(r),g_1(r)).$$
If $\eps>0$, we fix $A_{\eps}>R_{\ell}$  such that 
\begin{equation}
\label{BB12}
 \int_{A_{\eps}}^{+\infty} |\partial_r Z_{\ell}|^mr^m\,dr +\|Z_{\ell}\|_{S(\{r\geq A_{\eps}+|t|\})}^m\leq \frac{\eps^m}{C},
\end{equation}
In this step, we prove that for all $\eps>0$, if $A>A_{\eps}$ satisfies
\begin{equation}
 \label{BB11} \int_{A}^{+\infty} \left( |\partial_r g_0|^m+|g_1|^m \right)r^m\,dr <\frac{\eps^m}{C}.
\end{equation}
Then
\begin{equation}
 \label{BB10'}\int_{A}^{+\infty} \left|\partial_r h_0\right|^m+|h_1|^m\,dr\leq \frac{\eps}{A^{m-1}}|h_0(A)|^m.
\end{equation} 
Fix $A>A_{\eps}$, let $(\tu_0,\tu_1)= \TTT_A(u_0,u_1)$, and 
$\tu$ the solution of the nonlinear wave equation \eqref{NLW} with initial data $(\tu_0,\tu_1)$ at $t=0$. Note that by \eqref{BB11} and small data theory, $\tu$ is global and scatters in both time directions. Note also that by our assumption, $\tu$ satisfies \eqref{BB2'}.
%

Define $\tg$  as the solution to the following equation
\begin{equation}
 \label{BB14}
 \left\{
 \begin{aligned}
 \partial_t^2\tilde{g}-\Delta \tilde{g}&=\indic_{\left\{r\geq A+|t|\right\}}\left(|\tu|^{2m}\tu-|Z_{\ell}|^{2m}Z_{\ell}\right)\\
 \vec{\tilde{g}}_{\restriction t=0}&=\TTT_A(g_0,g_1),
\end{aligned}
\right.
\end{equation} 
and $\tilde{g}_{\rm L}$ the solution of the free wave equation with the same initial data. Notice that $(\partial_t^2-\Delta)(\tilde{u}-Z_{\ell})=(\partial_t^2-\Delta)\tilde{g}$ for $r>A+|t|$ and $\vec{\tilde{g}}(0,r)=(\tilde{u}_0-Z_{\ell},\tilde{u}_1)(r)$ for $r>A$. Thus, by finite speed of propagation, $\tilde{g}=\tilde{u}-Z_{\ell}$ for $r>A+|t|$, and we can rewrite the first equation in \eqref{BB14}:
\begin{equation}
 \label{BB14'}
 \partial_t^2\tilde{g}-\Delta \tilde{g}=\indic_{\left\{r\geq A+|t|\right\}}\left(|Z_{\ell}+\tilde{g}|^{2m}(Z_{\ell}+\tilde{g})-|Z_{\ell}|^{2m}Z_{\ell}\right).
 \end{equation}
Using \eqref{BB14'}, Strichartz estimates and H\"older inequality, we see that for all time-interval $I$ containing $0$:
\begin{equation*}
\left\|\tilde{g}-\tilde{g}_{\rm L}\right\|_{S(I)}+  \sup_{t\in I_{\max}(u)}\left\|\vec{\tg}(t)-\vec{\tg}_{\rm L}(t)\right\|_{\LLL^m}\leq C\left(\|Z_{\ell}\|^{2m}_{S(\{r\geq A+|t|\})}\|\tilde{g}\|_{S(I)}+\|\tilde{g}\|_{S(I)}^{2m+1}\right).
\end{equation*}
By \eqref{BB11}, \eqref{BB12} and a straightforward bootstrap argument, we deduce that for all interval $I$ with $0\in I$, 
\begin{equation*}
 \|\tilde{g}\|_{S(I)}\leq C\|\tilde{g}_{\rm L}\|_{S(I)}\leq C\left\|\TTT_A(g_0,g_1)\right\|_{\LLL^m}\leq C\eps,
\end{equation*}
and 
\begin{equation}
\label{g_approx_gL}
 \sup_{t\in \R}\left\|\vec{\tilde{g}}(t)-\vec{\tilde{g}}_{\rm L}(t)\right\|_{\LLL^m}\leq C\eps^{2m}\left\|\TTT_A(g_0,g_1)\right\|_{\LLL^m}.
\end{equation}
By the exterior energy property \eqref{P6} in Proposition \ref{P:prelimLW}, the following holds for all $t\geq 0$ or for all $t\leq 0$:
\begin{multline*}
 \label{BB15}
 \int_A^{+\infty} \left( |h_0|^m+|h_1|^m \right)\,dr\leq C\int_{A+|t|}^{+\infty} |\partial_{t,r}\tilde{g}_{\rm L}|^mr^m\,dr 
 \\ \leq C\left(\left( \eps^{2m}\|\TTT_A(g_0,g_1)\|_{\LLL^m} \right)^m+\int_{A+|t|}^{+\infty} |\partial_{t,r}\tilde{g}|^mr^m\,dr\right),
\end{multline*} 
where at the last line we have used \eqref{g_approx_gL}.

Letting $t\to \pm \infty$ and using \eqref{BB2'}, we deduce
$$\int_{A}^{+\infty} |\partial_{r}h_0|^m+|h_1|^m\,dr\leq C\eps^{2m^2} \int_A^{+\infty} \left(|\partial_rg_0|^m+|g_1|^m\right)r^m\,dr.$$
The desired estimate \eqref{BB10'} follows, taking $\eps$ small and using \eqref{P4} in Proposition \ref{P:prelim}.

\EMPH{Step 4} Fix a small $\eps>0$ and let $A_{\eps}$ be as in Step 3, i.e. such that \eqref{BB12} holds.
In this step, we prove that $r\leq A_{\eps}$ on $\esssupp (u_0-Z_{\ell},u_1)$. 

Indeed, if not, we obtain from \eqref{BB10'} and that there exists $A>A_{\eps}$ such that $h_0(A)\neq 0$. 
Using a similar argument as in Step 1, we deduce from \eqref{BB10'} that for all $A\geq A_{\eps}$ such that \eqref{BB11} holds,
\begin{equation}
\label{BB16} 
\forall B\in [A,2A]\quad \left|h_0(A)-h_0(B)\right|\leq C\eps|h_0(A)|.
\end{equation} 
If $\esssupp(u_0-Z_{\ell},u_1)$ is not bounded, we deduce by \eqref{BB10'} that $h_0(A)\neq 0$ for all large $A>0$. If $\eps>0$ is small enough, we deduce using \eqref{BB16} that 
$$\lim_{r\to+\infty} r^{\alpha} h_0(r)=+\infty,$$
where $\alpha \in (0,2m-2)$ is fixed.
Since 
$$v_0(r)-\ell=h_0(r)-\ell+rZ_{\ell},$$
this contradicts \eqref{BB9} in Step 2 and the asymptotic estimate \eqref{Zell2} of $Z_{\ell}$.

We have proved that $\esssupp(u_0-Z_{\ell},u_1)$ is bounded. Using \eqref{BB10'}, \eqref{BB16} and a straightforward bootstrap argument, we deduce that $r\leq A_{\eps}$ on the support of $\esssupp (u_0-Z_{\ell},u_1)$. 

\EMPH{Step 5} Fix  a small $\eps>0$. We have proved in Step 4 that $(u_0,u_1)(r)=(Z_{\ell}(r),0)$ for almost every $r\geq A_{\eps}$, where $A_{\eps}$ depends only on $\ell$. We will prove $(u_0,u_1)(r)=(Z_{\ell}(r),0)$ for $r>R_{\ell}$, a contradiction with Proposition \ref{P:stat} since $(u_0,u_1)\in \LLL^m$. 

We argue by contradiction, assuming that there exists $B>R_{\ell}$ such that $B\in \esssupp(u_0-Z_{\ell},u_1)$. Using a similar argument as in Step 3, but on small time intervals (see e.g. the proof of Proposition 2.2 (a), \S 2.2.1 in \cite{DuKeMe13}), we prove that the following holds for all $t\geq 0$ or for all $t\leq 0$:
\begin{equation}
 \label{prop_supp}
 B+|t|\in \esssupp\Big((u(t)-Z_{\ell},\partial_tu(t))\Big).
\end{equation} 
Choose $t_0$ such that 
\begin{equation}
\label{choice_t0}
B+|t_0|>A_{\eps} \text{ on }\esssupp\Big((u(t_0)-Z_{\ell},\partial_tu(t_0))\Big).
\end{equation} 
It is easy to see that $u$ satisfies the following:
for all $A>|t_0|$ the solution $\tilde{u}$ of
\begin{equation*}
\partial_t^2\tilde{u}-\Delta \tilde{u}=\iota |\tu|^{2m}\tu \indic_{\{r\geq A+|t-t_0|\}}
\end{equation*}
with initial data $\TTT_A(\vec{u}(t_0))$ at $t=t_0$ is not a scattering solution, or is scattering and satisfies
\begin{equation*}
 \liminf_{t\to\pm\infty} \int_{A+|t-t_0|}^{+\infty} |\partial_{t,r} \tilde{u}(t,r)|^mr^m\,dr=0.
\end{equation*} 
We can then go through Step 1, \ldots, Step 4 above, but with initial data at $t=t_0$, and restricting to $r>|t_0|$. Note that by finite speed of propagation, the limit $\ell$ obtained in Step 2 for $t=0$ and for $t=t_0$ is the same, i.e.
$$\lim_{r\to +\infty} r u(t_0,r)=\lim_{r\to+\infty} ru(0,r).$$
By the conclusion of Step 4, we obtain that $r<\max(A_{\eps},t_0)$ on $\esssupp(\vec{u}(t_0)-Z_{\ell})$, contradicting \eqref{choice_t0}. The proof is complete.

\qed

\section{Dispersive term}
\label{S:dispersive}
This section concerns the existence of a ``dispersive'' component for a solution $u$ of \eqref{NLW} that remains bounded in $\LLL^m$ along a sequence of times. This component is the strong limit of $\vec{u}(t)$, in $\LLL^m$, outside the origin in the finite time blow-up case (see Subsection \ref{SS:regular}) , and a solution of the linear wave equation in the global case (see Subsection \ref{SS:radiation}).
\subsection{Regular part in the finite time blow-up case}
\label{SS:regular}
\begin{proposition}
 \label{P:regular}
Let $u$ be a radial solution of \eqref{NLW}, \eqref{ID}. Assume 
 $$T_+(u)<\infty,\quad \liminf_{t\to T_+(u)} \|\vec{u}(t)\|_{\LLL^m}<\infty.$$
Then there exists a solution $v$ of \eqref{NLW}, defined in a neighborhood of $t=T_+$, such that for all $t$ in $I_{\max}(u)\cap I_{\max}(v)$,
$$ \forall r>T_+-t,\quad \vec{u}(t,r)=\vec{v}(t,r).$$
\end{proposition}
We omit the proof (see Subsection 6.3 in \cite{DuRoy15P} for a very close proof).
\subsection{Extraction of the radiation term in the global case}
\label{SS:radiation}
We prove here:
\begin{proposition}
\label{P:radiation}
 Let $u$ be a radial solution of \eqref{NLW}, \eqref{ID}. Assume 
 $$T_+(u)=+\infty,\quad \liminf_{t\to +\infty} \|\vec{u}(t)\|_{\LLL^m}<\infty.$$
 Then there exists a solution $v_{\rm L}$ of the free wave equation \eqref{LW} such that for all $A\in \R$,
\begin{equation}
\label{radiation_A}
 \lim_{t\to+\infty} \int_{|x|\geq A+|t|}\left( |\partial_t(u-v_{\rm L})|^m+|\partial_r(u-v_{\rm L})|^m \right)r^m\,dr=0.
\end{equation} 
\end{proposition}
The proof relies on the following lemma, which is a consequence of finite speed of propagation, Strichartz estimates and the small data theory. We omit the proof, which is an easy adaptation of the proofs of Claim 2.3 and 2.4 in \cite{DuKeMe16Pa} where the usual energy is replaced by the $\LLL^m$-energy:
\begin{lemma}
 \label{L:FSP}
 There exists $\eps_1>0$ with the following property.
 Let $u$ be a solution of \eqref{NLW}, \eqref{ID} such that $T_+(u)=+\infty$. Let $T\geq 0$ and $A\geq 0$. 
 Assume 
  $\left\|S_{\rm L}(\cdot-T)\vec{u}(T)\right\|_{S(\left\{|x|\geq A+t,\; t\geq T\right\})}=\eps'<\eps_1.$
  Then
  $\|u\|_{S(\left\{|x|\geq A+t,\; t\geq T\right\})}\leq 2\eps'$, and there exists a solution $v_{\rm L}$ of the linear wave equation 
 such that \eqref{radiation_A} holds.
\end{lemma}
\begin{proof}[Proof of Proposition \ref{P:radiation}]
(see also Subsection 3.3 in \cite{DuKeMe13}).
\EMPH{Step 1} 
Let $t_n\to +\infty$ such that the sequence $(\vec{u}(t_n))_n$ is bounded in $\LLL^m$. In this step we prove that there exists $\delta>0$ such that for large $n$,
\begin{equation}
 \label{small_Snorm}
 \left\|S_{\rm L}(\cdot) \vec{u}(t_n)\right\|_{S\left(\left\{ |x|\geq (1-\delta)t_n +t,\; t\geq 0\right\}\right)}<\eps_1,
\end{equation} 
where $\eps_1$ is given by Lemma \ref{L:FSP}. We argue by contradiction, assuming (after extraction of subsequences) that there exists a sequence $\delta_n\to 0$ such that 
\begin{equation}
 \label{large_Snorm}
 \left\|S_{\rm L}(\cdot) \vec{u}(t_n)\right\|_{S\left(\left\{ |x|\geq (1-\delta_n)t_n +t,\; t\geq 0\right\}\right)}\geq \eps_1,
 \end{equation} 
 Extracting subsequences again, we can assume that the sequence $\left(\vec{u}(t_n)\right)_n$ has a profile decomposition with profiles $U^j_{\rm L}$ and parameters $(\lambda_{j,n},t_{j,n})_n$. Let $J$ be a large integer such that 
 $$ \left\|S_{\rm L}(\cdot)\left( \vec{u}(t_n)-\sum_{j=1}^J \vec{U}_{{\rm L},n}^j(0) \right)\right\|_{S(\R)}\leq \frac{\eps_1}{2}.$$
A contradiction will follow if we prove (possibly extracting subsequences in $n$), that for all $j\in \{1,\ldots,J\}$,
\begin{equation}
 \label{Snorm_Uj_0}
\lim_{n\to\infty} \left\|S_{\rm L}(\cdot)\vec{U}_{{\rm L},n}^j(0) \right\|_{S\left(\left\{ |x|\geq (1-\delta_n)t_n +t,\; t\geq 0\right\}  \right)}=0.
 \end{equation} 
 We have 
 \begin{equation*}
  \left\|S_{\rm L}(\cdot)\vec{U}_{{\rm L},n}^j(0) \right\|_{S\left(\left\{ r\geq (1-\delta_n)t_n +t,\; t\geq 0\right\}  \right)}=\left\|U_{\rm L}^j\right\|_{S(A_{j,n})},
 \end{equation*}
where 
$$A_{j,n}:=\left\{(t,r)\in \R\times (0,\infty)\;:\; t\geq -\frac{t_{j,n}}{\lambda_{j,n}}\text{ and } r\geq \frac{(1-\delta_n)t_n}{\lambda_{j,n}}+\left|t+\frac{t_{j,n}}{\lambda_{j,n}}\right|\right\}.$$ 
As a consequence, we see that we can extract subsequences so that the characteristic function of $A_{j,n}$ goes to $0$ pointwise unless $t_{j,n}/\lambda_{j,n}$ and $t_n/\lambda_{j,n}$ are bounded. Time translating the profile $U^j_{\rm L}$ and extracting again, we can assume:
$$ \lim_{n\to\infty} t_n/\lambda_{j,n}=\tau_0\in [0,\infty),\quad \forall n,\; t_{j,n}=0.$$
By finite speed of propagation and the small data theory,
\begin{equation}
\label{FSPSDT}
\lim_{A\to+\infty} \limsup_{n\to+\infty} \int_{|x|\geq t_n+A}\left|r \partial_{r,t} u(t_n)\right|^m\,dr=0. 
\end{equation} 
By Proposition \ref{P:exterior_profiles}, for all $A\in \RR$, we have that for large $n$,
\begin{align*}
\int_{t_n+A}^{+\infty} \left|r(\partial_{r,t}u(t_n))\right|^m\,dr &\geq \frac{1}{2} \int_{t_n+A}^{+\infty}\left|r(\partial_{r,t}U^j_{L,n}(0))\right|^m\,dr\\ 
&=\frac{1}{2} \int_{\frac{t_n+A}{\lambda_{j,n}}}^{+\infty}\left|r(\partial_{r,t}U^j_{\rm L}(0))\right|^m\,dr\\
&\underset{n\to\infty}{\longrightarrow}\frac 12\int_{\tau_0}^{+\infty}\left|r(\partial_{r,t}U^j_{\rm L}(0))\right|^m\,dr.
\end{align*} 
Combining with \eqref{FSPSDT}, we see that if $U^j_{\rm L}$ is not identically $0$, then $\tau_0$ is strictly positive, and we can rescale the profile $U^j_{\rm L}$ to assume $\tau_0=1$, and $\lambda_{j,n}=t_n$. Using \eqref{FSPSDT} we see that $\esssupp \vec{U}^j_{\rm L}(0)$ is included in the unit ball of $\R^3$, which implies
$$\left\|U_{\rm L}^j\right\|_{S(A_{j,n})}=\left\|U^j_{\rm L}\right\|_{S\left(\left\{t\geq 0,\; r\geq (1-\delta_n)+t\right\}\right)}\underset{n\to\infty}{\longrightarrow}0,$$
concluding the proof of \eqref{Snorm_Uj_0} in this case. Step 1 is complete.

\EMPH{Step 2} By Step 1 and Lemma \ref{L:FSP}, for all $A\in \R$, there exists a solution $v_{\rm L}^{A}$ of the free wave equation such that 
\begin{equation}
\label{CVvLA}
 \lim_{t\to+\infty} \int_{|x|\geq A+|t|}\left( |\partial_t(u-v_{\rm L}^A)|^m+|\partial_r(u-v_{\rm L}^A)|^m \right)r^m\,dr=0.
\end{equation} 
We consider the sequence $t_n\to+\infty$ of Step 1 and assume, extracting a subsequence if necessary, that $\vec{u}(t_n)$ has a profile decomposition $\left( U^j_{\rm L},\left( \lambda_{j,n},t_{j,n} \right)_n \right)_{j\geq 1}$. Reordering the profiles and rescaling and time translating $U_{\rm L}^1$ if necessary, we can assume, without loss of generality, that $t_{1,n}=t_n$ and $\lambda_{1,n}=1$ for all $n$. In other words, $\vec{U}_{\rm L}^1(0)$ is the weak limit, as $n$ goes to infinity, of $\vec{S}_{\rm L}(-t_n)\vec{u}(t_n)$.  Note that $U_{\rm L}^1$ might be identically $0$.

Fix $A\in \RR$. Then
$$ \vec{u}(t_n)-\vec{v}_{\rm L}^A(t_n)=\vec{U}^1_{\rm L}(t_n)-\vec{v}_{\rm L}^A(t_n)+\sum_{j=2}^J \vec{U}^j_{{\rm L},n}(0)+\vec{w}^J_n(t_n),$$
i.e. $\vec{u}(t_n)-\vec{v}_{\rm L}^A(t_n)$ has a profile decomposition  $\left( \widetilde{U}^j_{\rm L},\left( \lambda_{j,n},t_{j,n} \right)_n \right)_{j\geq 1}$, with $\tU^j_{\rm L}=U^j_{\rm L}$ if $j\geq 2$, and $\tU^1_{\rm L}=U^1_{\rm L}-v^A_{\rm L}$. By Proposition \ref{P:exterior_profiles},
$$\limsup_{n\to\infty}\int_{r\geq A+t_n} \left|r\partial_{r,t} (u-v_{\rm L}^A)(t_n)\right|^m\,dr\geq \limsup_{n\to\infty} \int_{r\geq t_n+A} \left|r\partial_{r,t}(U^1_{\rm L}-v_{\rm L}^A)(t_n)\right|^m,$$
and thus, by \eqref{CVvLA}
$$ \lim_{n\to\infty} \int_{r\geq t_n+A} \left|r\partial_{r,t}(U^1_{\rm L}-v_{\rm L}^A)(t_n)\right|^m=0.$$
Using \eqref{CVvLA} again, we obtain
$$ \lim_{n\to\infty} \int_{r\geq t_n+A} \left|r\partial_{r,t}(U^1_{\rm L}-u)(t_n)\right|^m=0.$$
This is valid for all $A\in \RR$. A simple argument using finite speed of propagation and small data theory yields 
$$ \lim_{t\to\infty} \int_{r\geq t+A} \left|r\partial_{r,t}(U^1_{\rm L}-u)(t)\right|^m=0.$$
Concluding the proof of the proposition with $v_{\rm L}=U^1_{\rm L}$.
\end{proof}
\section{Scattering/blow-up dichotomy}
\label{S:endofproof}
In this section we prove Theorem \ref{T:scattering}. Let $u$ be a solution of \eqref{NLW} such that 
\begin{equation}
 \label{T1}
\liminf_{t\to T_+(u)}\|\vec{u}(t)\|_{\LLL^m}<\infty.
\end{equation} 
We must prove:
\begin{enumerate}
 \item \label{I:global} $T_+(u)=+\infty$;
\item \label{I:scattering} if $T_+(u)=+\infty$, then $u$ scatters to a linear solution in $\LLL^m$.
\end{enumerate}
The proofs of \eqref{I:global} and \eqref{I:scattering} are very similar, and are a simplified version of the corresponding proofs in \cite{DuRoy15P}. We will only sketch the proof of \eqref{I:scattering} and explain the necessary modification to obtain \eqref{I:global}.

\subsection{Proof of scattering}
\label{SS:scattering}
Let $u$ be a global solution and let $t_n\to +\infty$ such that $\vec{u}(t_n)$ is bounded. Let $v_{\rm L}$ be the linear component of $u$, given by Proposition \ref{P:radiation}. Extracting subsequences, we can assume that $(\vec{u}(t_n)-\vec{v}_{\rm L}(t_n))_n$ has a profile decomposition with profiles $U^j_{\rm L}$ and parameters $(\lambda_{j,n},t_{j,n})_n$. As before, we denote by $U^j_{{\rm L},n}$ the modulated profiles (see \eqref{ULnJ}). Extracting subsequences and translating the profiles in time if necessary, one of the following three cases holds.

\EMPH{Case 1} 
\begin{equation}
 \label{HCase1}
\forall j\geq 1, \quad U^j_{\rm L}\equiv 0\text{ or }\lim_{n\to\infty} \frac{-t_{j,n}}{\lambda_{j,n}}=-\infty.
 \end{equation} 
Let $T\gg 1$ such that $\|v_{\rm L}\|_{S((T,+\infty))}<\delta_0/2$, where $\delta_0$ is given by the small data theory (see Proposition  \ref{prop:LWP}). By \eqref{HCase1}, for all $j$, 
$$\lim_{n\to\infty}\|U^j_{{\rm L},n}\|_{S((T-t_n,0))}=\lim_{n\to\infty}\|U^j_{\rm L}\|_{S\left( \left(\frac{T-t_n-t_{j,n}}{\lambda_{j,n}},\frac{-t_{j,n}}{\lambda_{j,n}}\right)\right)}=0.$$
Thus for large $n$, 
$$\|S_{\rm L}(\cdot)\vec{u}(t_n)\|_{S((T-t_n,0))}<\delta_0.$$
By Proposition \ref{prop:LWP}, for large $n$,
$$ \|u\|_{S(T,t_n)}=\|u(t_n+\cdot)\|_{S(T-t_n,0)}<2\delta_0.$$
Letting $n\to\infty$, we deduce $\|u\|_{S(T,+\infty)}<2\delta_0$, and thus $u$ scatters.

\EMPH{Case 2}
We assume
\begin{equation}
 \label{HCase2}
\forall j\geq 1, \quad U^j_{\rm L}\equiv 0\text{ or }\lim_{n\to\infty} \frac{-t_{j,n}}{\lambda_{j,n}}\in\{\pm\infty\}.
 \end{equation} 
and
\begin{equation}
 \label{HCase2b}
 \exists j_0\geq 1,\quad U^{j_0}_{\rm L}\not\equiv 0\text{ and }\lim_{n\to\infty} \frac{-t_{j_0,n}}{\lambda_{j_0,n}}=+\infty.
\end{equation} 
We will use a channel of energy argument based on the following observation, which is a direct consequence of the explicit form of the solution (see \eqref{3}, \eqref{2}):
\begin{claim}
\label{C:ext}
 Let $u_{\rm L}$ be a nonzero solution of the linear wave equation \eqref{LW} with initial data in $\LLL^m$. Then there exists $A\in \RR$ such that 
 $$\liminf_{t\to+\infty}\int_{A+t}^{+\infty} r^m|\partial_{r,t}u_{\rm L}|^m\,dr>0.$$
\end{claim}
If $j\geq 1$, we have 
$$\|U^{j}_{{\rm L},n}\|_{S(\left\{t\geq 0,\; r\geq t\right\})}=\|U^{j}_{\rm L}\|_{S\left(\left\{t\geq -\frac{t_{j,n}}{\lambda_{j,n}},\;r\geq t+\frac{t_{j,n}}{\lambda_{j,n}}\right\}\right)}.$$
Noting that under the assumptions of Case 2, 
$$\forall j\geq 1,\quad \indic_{\left\{t\geq -\frac{t_{j,n}}{\lambda_{j,n}},\;r\geq t+\frac{t_{j,n}}{\lambda_{j,n}}\right\}}\underset{n\to\infty}{\longrightarrow}0$$
pointwise, otherwise $U^j_{\rm L}\equiv 0$. We obtain
$$\forall j\geq 1,\quad \lim_{n\to\infty}\|U^j_{{\rm L},n}\|_{S\left(\left\{r\geq t\geq 0\right\}\right)}=0$$
and thus
$$\lim_{n\to\infty}\left\|S_{\rm L}(t)\vec{u}(t_n)\right\|_{S\left(\left\{r\geq t\geq 0\right\}\right)}=0.$$ 
By the small data theory (see Proposition \ref{prop:LWP}) and finite speed of propagation
\begin{multline}
\label{utn_linear}
\lim_{n\to\infty}
\bigg(  \left\|u(t_n+\cdot)\right\|_{S\left(\left\{r\geq t\geq 0\right\}\right)}\\
+\sup_{t\geq 0} \int_t^{+\infty}\left|\partial_{t,r}(u(t_n+t)-S_{\rm L}(t)\vec{u}(t_n))\right|^mr^m\,dr\bigg)=0. 
\end{multline} 
Let $j_0$ be as in \eqref{HCase2b}. By Claim \ref{C:ext}, there exists $A\in \RR$ such that
$$\liminf_{t\to +\infty} \int_{\lambda_{j_0,n}A-t_{j_0,n}+t}^{\infty} \big|r\partial_{t,r}U^{j_0}_{{\rm L},n}\big|^{m}\,dr>0.$$
For large $n$, $\lambda_{j_0,n}A-t_{j_0,n}\geq 0$. By Proposition \ref{P:exterior_profiles}, we deduce from \eqref{utn_linear} that for large $n$,
$$\liminf_{t\to+\infty} \int_{\lambda_{j_0,n}A-t_{j_0,n}+t-t_{n}}^{\infty} r^m|\partial_{t,r}(u(t,r)-v_{\rm L}(t,r))|^m\,dr>0,$$
contradicting the definition of $v_{\rm L}$.

\EMPH{Case 3} In this last case we assume 
\begin{equation}
\label{HCase3}
\exists j\geq 1,\quad \forall n,\quad t_{j,n}=0\text{ and }U^j_{\rm L}\not\equiv 0.
\end{equation} 
This is the core of the proof, where we use Proposition \ref{P:BB1}, and thus the fact that equation \eqref{NLW} has no nonzero stationary solution in $\LLL^m$.

We will use Subsection \ref{SS:approx} to approximate $u$, outside appropriate wave cones, by a sum of profiles. As in Subsection \ref{SS:approx}, we let $\JJJ_0$ be the set of indices $j$ such that $t_{j,n}=0$ for all $n$ and  $\JJJ_{\infty}$  the set of $j$ such that $t_{j,n}/\lambda_{j,n}$ goes to $+\infty$ or $-\infty$. Extracting subsequences and translating the profiles in time if necessary, we can assume $\NN\setminus\{0\}=\JJJ_0\cup \JJJ_{\infty}$.
Let $\delta_1>0$ be a small number, smaller than the number given by the small data theory, and such that there exists $j\in \JJJ_0$ with 
$\|\vec{U}^j_{\rm L}(0)\|_{\LLL^m}>\delta_1$. We let $j_0\in \JJJ_0$ such that $\|\vec{U}^{j_0}_{\rm L}(0)\|_{\LLL^m}>\delta_1$, and 
\begin{equation}
\label{cond_j0}
\left(j\in \JJJ_0\text{ and } \|\vec{U}^j_{\rm L}(0)\|_{\LLL^m}>\delta_1\right)\Longrightarrow \lim_{n\to\infty} \frac{\lambda_{j_0,n}}{\lambda_{j,n}}=+\infty. 
\end{equation} 
We note that by Proposition \ref{P:Pythagorean}, there exists a finite number of $j\in \JJJ_0$ with 
$\|\vec{U}^j_{\rm L}(0)\|_{\LLL^m}>\delta_1$, so that (in view of the pseudo-orthogonality property \eqref{P-O}) $j_0$ is well-defined. 
By Proposition \ref{P:BB1}, there exist $A,\,\eta>0$, $U^{j_0}\in S(\RR)$ such that $\vec{U}^{j_0}\in C^0(\RR,\LLL^m)$,
\begin{equation}
\label{eq:Uj0}
(\partial_t^2-\Delta)U^{j_0}=\iota |U^{j_0}|^{2m}U^{j_0}\indic_{\{r\geq A+|t|\}},\quad \vec{U}^{j_0}(0)=\TTT_A(\vec{U}^{j_0}_{\rm L}(0)),
\end{equation}
and the following holds for all $t\geq 0$ or for all $t\leq 0$
\begin{equation}
\label{channel3}
\int_{|t|+A}^{+\infty} |r\partial_{t,r}U^{j_0}|^{m}\,dr\geq \eta.
\end{equation} 
Note that $(U_{\rm L}^j,\lambda_{j,n},t_{j,n})_{j\geq 0}$
with $U^0_{\rm L}=v_{\rm L}$ and $\lambda_{0,n}= 1, t_{0,n}=t_n$  is a profile decomposition of $\vec{u}(t_n)$.
According to Lemma \ref{L:approx}, 
\begin{equation}
 \label{approx3}
u(t+t_n)=v_{\rm L}(t+t_n)+\sum_{j=1}^J U^j_n(t)+w_n^J(t)+\eps_n^J(t),\quad t \in[-t_n,+\infty)
 \end{equation} 
where the modulated profiles $U^j_n$ for $j\neq j_0$ are defined in Subsection \ref{SS:approx} and 
\begin{equation*}
\limsup_{n\to\infty}\left(\|\eps_n^J\|_{S(\{t\in [-t_n,+\infty),\;r>A\lambda_{j_0,n}+|t|\})}+\sup_{t\geq -t_n}\int_{|t|+A\lambda_{j_0,n}}^{+\infty} \left|r\partial_{t,r}\eps_n^J(t,r)\right|^{m}\,dr\right)
\end{equation*}
goes to $0$ as $J$ goes to infinity.
 It can be deduced from Proposition \ref{P:exterior_profiles} that  for all sequence $(\theta_n)_n$ in $[-t_n,+\infty)$,
\begin{multline}
\label{exterior_NL}
o_n(1)+\int_{A\lambda_{j_0,n}+|\theta_n|}^{+\infty}|r\partial_{r,t} (u-v_{\rm L})(t_n+\theta_n,r)|^m\,dr\\
\geq \int_{A\lambda_{j_0,n}+|\theta_n|}^{+\infty}\left|r\partial_{r,t} U_{n}^{j_0}(\theta_n,r)\right|^m\,dr.
\end{multline}
Indeed, this can be proved by noticing that \eqref{approx3} (and its time derivative) at $t=\theta_n$ can be considered as a profile decomposition of the sequence $\big((\vec{u}-\vec{v}_{\rm L})(\theta_n+t_n)\big)_n$ and using Proposition \ref{P:exterior_profiles} and finite speed of propagation. We refer to the proof of (3.18) in \cite{DuRoy15P} for a detailed proof in a very similar setting.

If \eqref{channel3} holds for $t\geq 0$, then by \eqref{exterior_NL}, for large $n$,
$$\limsup_{t\to+\infty} \int_{t+A\lambda_{j_0,n}}^{+\infty} |r\partial_{t,r}(u(t_n,r)-v_{\rm L}(t_n,r))|^{m}\,dr\geq \frac{\eta}{2},$$
contradicting the definition of $v_{\rm L}$. 

If \eqref{channel3} holds for $t\leq 0$, we use \eqref{exterior_NL} at $\theta_n=-t_n$ together with \eqref{channel3} and obtain that for large $n$ 
$$\int_{t_n+A\lambda_{j_0,n}}^{+\infty} |r\partial_{t,r}(u(0,r)-v_{\rm L}(0,r))|^{m}\,dr\geq \frac{\eta}{2},$$
a contradiction since $\vec{u}(0)\in \LLL^m$. The proof is complete.

\subsection{Proof of global existence}
\label{SS:global}
We argue by contradiction, assuming that \eqref{T1} holds and that $T_+=T_+(u)$ is finite. Let $v$ be the regular part of $u$ at $t=T_+$, defined by Proposition \ref{P:regular}. Recall that $v$ is a solution of \eqref{NLW} defined in a neighborhood of $T_+(u)$ and such that
\begin{equation}
\label{cond_v}
\forall t\in I_{\max}(u)\cap I_{\max}(v),\; \forall r>T_+-t,\quad \vec{u}(t,r)=\vec{v}(t,r). 
\end{equation} 
As in Subsection \ref{SS:scattering}, we consider a sequence $t_n\to T_+$ such that $\vec{u}(t_n)$ is bounded in $\LLL^m$, and we assume (extracting subsequences if necessary) that $\left(\vec{u}(t_n)-\vec{v}(t_n)\right)_n$ has  a profile decomposition with profiles $U^j_{\rm L}$ and parameters $(\lambda_{j,n},t_{j,n})_n$. We distinguish again between three cases.

\EMPH{Case 1} We assume \eqref{HCase1}. By the same proof as in Case 1 of Subsection \ref{SS:scattering}, we obtain:
$$\lim_{n\to\infty} \left\|S_{\rm L}(\cdot)\left(\vec{u}(t_n)-\vec{v}(t_n)\right)\right\|_{S((-\infty,0))}=0.$$
By Lemma \ref{L:ELTPT}, if $T<T_+(u)$ is in the domain of definition of $v$, close to $T_+(u)$,
$$\lim_{n\to\infty} \left\|\vec{u}(t_n)\right\|_{S((T,t_n))}<\infty,$$
which contradicts the blow-up criterion:
$$\left\|\vec{u}(t_n)\right\|_{S\big((T,T_+(u))\big)}=+\infty,$$

\EMPH{Case 2} We assume \eqref{HCase2} and \eqref{HCase2b}.
Fix $j_0\geq 1$ such that \eqref{HCase2b} holds.
Using Claim \ref{C:ext} and an argument very similar to the one of Case 2 of Subsection \ref{SS:scattering}, we obtain that for large $n$,
$$ \liminf_{t\to T_+(u)} \int_{r\geq A\lambda_{j_0,n}-t_{j_0,n}+t-t_n}|\partial_{t,r}(u-v)(t,r)|^m\,r^mdr>0,$$
where $A\in \RR$ is given by Claim \ref{C:ext},
contradicting \eqref{cond_v} (since for large $n$, $A\lambda_{j_0,n}-t_{j_0,n}\geq 0$).

\EMPH{Case 3} We assume \eqref{HCase3}. We define $\JJJ_0$, $\JJJ_{\infty}$ as in Case 3 of Subsection \ref{SS:scattering} and choose $j_0\in \JJJ_0$ such that \eqref{cond_j0} holds. Using Proposition \ref{P:BB1}, we obtain $A,\,\eta>0$, and a solution $U^{j_0}\in S(\RR)$ of \eqref{eq:Uj0}, such that \eqref{channel3} holds for all $t\geq 0$ or for all $t\leq 0$. We distinguish two cases. 

If \eqref{channel3} holds for all $t\geq 0$, then we prove using Lemma \ref{L:ELTPT} and Proposition \ref{P:exterior_profiles} that for large $n$,
$$ \liminf_{t\to T_+(u)} \int_{r\geq A\lambda_{j_0,n}+t-t_n}|\partial_{t,r}(u-v)(t,r)|^m\,r^mdr>0,$$
a contradiction with \eqref{cond_v}.

If \eqref{channel3} holds for all $t\leq 0$, we let $T\in [0,T_+(u))$ such that $T$ is in the domain of definition of $v$. Using Lemma \ref{L:ELTPT} and Proposition \ref{P:exterior_profiles}, we deduce that for large $n$:
$$\int_{r\geq A\lambda_{j_0,n}+t_n-T}|\partial_{t,r}(u-v)(T,r)|^m\,r^mdr\geq \frac{\eta}{2},$$
a contradiction for large $n$, since $\partial_{t,r}(u-v)(T,r)$ is supported in $|x|\leq T_+-T$. This concludes the sketch of proof.

\appendix
\section{Proof of Proposition \ref{prop:fs}}
	\EMPH{The ``only if '' part}
	First of all, we have a sequence of smooth radial functions $(f_{n})_{n}$
	with compact supports, such that
	\begin{equation}
	\label{eq:fs-3}
	\int_{0}^{+\infty}|\partial_{r}(f-f_{n})(r)|^{m}r^{m}dr\longrightarrow 0,\,n\rightarrow\infty\,.
	\end{equation}
	As a consequence, we clearly have \eqref{eq:fs-1}.
	Notice that 
	for $0<r<r'<+\infty$, we have
	\[
	|f(r')-f(r)|\leq \frac{C_{m}}{r^{\frac1m}}
	\left(\int_{r}^{r'}|s\partial_{s}f(s)|^{m}ds\right)^{1/m}\,,
	\]
	and this yields that $f(r)$ is continuous.
	
	To see \eqref{eq:fs-2}, we first prove:
	\[|f(r)|\leq \frac{1}{r^{1/m}}
	\left(\int_{r}^{+\infty}|s\partial_{s}f(s)|^{m}ds\right)^{1/m}\,,\]
	and 
	\[
	|rf(r)|\leq r^{\frac{m-1}{m}}
	\left(\int_{0}^{r}|\partial_{s}(sf(s))|^{m}ds\right)^{1/m}\,.
	\]
indeed, if $f\in C_{0}^{\infty}((0,+\infty))$, then the preceding inequality follows from the fundamental theorem of calculus and H\"older inequality. The case of a general function $f$ can be deduced from \eqref{eq:fs-3}. The desired estimate \eqref{eq:fs-2} is an immediate consequence of these two inequalities.

	
	\EMPH{The`` if '' part} Given a radial function $f(x)$ on $\mathbb R^{3}$,
	satisfying the conditions \eqref{eq:fs-1}\eqref{eq:fs-2}, we are to construct a sequence of smooth
	radial functions $f_{n}(x)$ compactly supported in $\mathbb R^{3}$ such that  \eqref{eq:fs-3} holds.
	\medskip
	
	To achieve this, we take a smooth radial function $\varphi(x)$ on 
	$\mathbb R^{3}$, such that $\varphi(x)=1$ for $|x|\leq 1$ and 
	$\varphi(x)=0$ if $|x|\geq 2$. Let $(\eps_{n})_{n}$ be a sequence
	of positive numbers, tending to zero as $n\rightarrow\infty$.
	Define 
	\begin{equation}
	\label{eq:fn}
	f_{n}(x)=\varphi(\eps_{n}x)
	\left(1-\varphi\left(\frac{x}{\eps_{n}}\right)\right)
	\bigl(f*\zeta_{\eps_{n}}\bigr)(x)\,,
	\end{equation}
	where $\zeta_{\eps}(\varrho)$ is the usual approximate delta function
	supported in $-\eps/2<\varrho<0$ and  $f*\zeta_{\eps}$
	denotes the radial convolution as in \cite{St77}, namely
	\[
	f*\zeta_{\eps}(x)=\int_{-\eps/2}^{0}\zeta_{\eps}(\varrho)\,
	f(|x|-\varrho)\,d\varrho\,.
	\]
	Then it is clear that $f_{n}(x)$ is smooth, radial and supported in $\{x\in \mathbb R^{3}\mid \eps_{n}\leq |x|\leq 2/\eps_{n}\}$.
	We have
	\begin{align}
	\nonumber&\partial_{r}\left(f(r)-f_{n}(r)\right)\\
	\label{a5}=&-\eps_{n}(\partial_{r}\varphi)(\eps_n r)
	\left(1-\varphi\left(\frac{r}{\eps_{n}}\right)\right)f(r)\\
	\label{a6}&+\varphi(\eps_{n}r)\frac{1}{\eps_{n}}(\partial_{r}\varphi)
	\left(\frac{r}{\eps_{n}}\right)f(r)\\
	\label{a7}&+\left[1-\varphi(\eps_{n}r)
	\left(1-\varphi\left(\frac{r}{\eps_{n}}\right)\right)\right]\partial_{r}f(r)\\
	\label{a8}&+\partial_{r}\left[\varphi(\eps_{n}r)
	\left(1-\varphi\left(\frac{r}{\eps_{n}}\right)\right)\right]
	\int_{-\eps_{n}/2}^{0}\Bigl(f(r)-f(r-\varrho)\Bigr)\,\zeta_{\eps_{n}}(\varrho)\,d\varrho\\
	\label{a9}&+\varphi(\eps_{n}r)
	\left(1-\varphi\left(\frac{r}{\eps_{n}}\right)\right)
	\int_{-\eps_{n}/2}^{0}\Bigl(\partial_{r}f(r)-\partial_{r}f(r-\varrho)\Bigr)\,\zeta_{\eps_{n}}(\varrho)\,d\varrho\,.
	\end{align}
	In view of \eqref{eq:fs-1}, one easily sees that 
	multiplying  by $r$ on both sides of the above identity,
	raising them to the power $m$ and integrating over $(0,+\infty)$, 
	we have the contributions of \eqref{a8}\eqref{a9} go to zero
	as $n\rightarrow \infty$. In fact, this is immediate for \eqref{a9} in view of
	the boundedness of $\varphi$ and the fact that $\zeta_{\eps}$
	is an approximation of the identity. For \eqref{a8}, we need to
	estimate two terms produced correspondingly by the cases when $\partial_{r}$
	hits on $\varphi(\eps_{n}r)$ and $\varphi(r/\eps_{n})$.
	In the first case, we use the fundamental theorem of calculus to write
	\[
	f(r)-f(r-\varrho)=\int_{0}^{1}\varrho\,\partial_{r}f(r-\theta\varrho)\,d\theta\,.
	\]
	Applying Minkowski's inequality, we are led to 
	estimating
	\begin{align*}
	\eps_{n}
	\int_{-\eps_{n}/2}^{0}\int_{0}^{1}
	\left(\int_{1/2\eps_{n}}^{4/\eps_{n}}|r\partial_{r}f(r)|^{m}dr\right)^{1/m}
	|\varrho\,\zeta_{\eps_{n}}(\varrho)|\,d\theta d\varrho
	\end{align*}
	which is clearly tending to zero as $n\rightarrow \infty$.
	A similar argument applies to the second case.
	In fact,  applying the same trick will lead us to estimating
	\[
	\int_{-\eps_{n}/2}^{0}\int_{0}^{1}\left(\int_{\eps_{n}/2}^{2\eps_{n}}
	|r\partial_{r}f(r)|^{m}dr\right)^{1/m}
	\frac{|\varrho|}{\eps_{n}}\,\zeta_{\eps_{n}}(\varrho)\,d\theta \,d\varrho\,,
	\]
	which tends to zero as $n\rightarrow \infty$.
    \medskip
    
	Next, by
	invoking \eqref{eq:fs-2}, one sees that the contribution from \eqref{a5} is bounded by 
	\[
	\left(\sup_{\frac{1}{\eps_{n}}\leq r\leq \frac{2}{\eps_{n}}}
	r^{\frac1m}\bigl|f(r)|\right)^{m}\cdot
	\int_{1/\eps_{n}}^{2/\eps_{n}}|\varphi'(\eps_{n}r)|^{m}\eps_{n}^mr^{m-1}dr\rightarrow 0\,,\,n\rightarrow\infty\,.
	\]
	Similar argument applies to \eqref{a6}
	thanks to \eqref{eq:fs-2}. Finally, the contribution of \eqref{a7}  is easily seen to be bounded by 
	\[\int_{0}^{2\eps_{n}}|r\partial_{r}f(r)|^{m}dr+
	\int_{1/\eps_{n}}^{+\infty}|r\partial_{r}f(r)|^{m}dr\longrightarrow0,\;
	n\rightarrow \infty\,.\]
	The proof is complete.

\bibliographystyle{plain}
\bibliography{toto}
\end{document}